\documentclass[reqno,a4paper,twoside]{amsart}

\usepackage[cmyk,usenames,dvipsnames,svgnames,table]{xcolor}

\definecolor{myblue}{cmyk}{0.77, 0.23, 0.0, 0.37}
\definecolor{mycyan}{cmyk}{0.46, 0.0, 0.0, 0.22}
\definecolor{myyellow}{cmyk}{0.0, 0.12, 0.6, 0.0}
\definecolor{mydarkgray}{cmyk}{0.1, 0.0, 0.2, 0.78}
\definecolor{myred}{cmyk}{0.0, 0.6, 0.61, 0.05}
\definecolor{myviolet}{cmyk}{0.0, 0.76, 0.15, 0.5}

\usepackage{titletoc}
\contentsmargin{2.55em}
\dottedcontents{section}[1.5em]{}{2.8em}{1pc}
\dottedcontents{subsection}[3.7em]{}{3.4em}{1pc}
\dottedcontents{subsubsection}[7.6em]{}{3.2em}{1pc}
\dottedcontents{paragraph}[10.3em]{}{3.2em}{1pc}

\usepackage{enumitem}
\setenumerate{label=\textnormal{(\arabic*)}}

\allowdisplaybreaks[3]

\usepackage{amsmath,amssymb,dsfont,mathrsfs,verbatim,bm,geometry}
\usepackage{mathtools, longtable, tabu, booktabs}
\usepackage[latin1]{inputenc}
\usepackage{array,hhline}
\usepackage[raggedright]{titlesec}
\usepackage{tikz}
\usepackage{float,subfig}
\usetikzlibrary{arrows, backgrounds, fit, matrix, positioning}
\usepackage{amsrefs}
\usepackage[foot]{amsaddr}

\titleformat{\section}
{\normalfont\Large\bfseries\center}{\thesection}{1em}{}
\titleformat{\subsection}
{\normalfont\large\bfseries}{\thesubsection}{1em}{}
\titleformat{\subsubsection}
{\normalfont\normalsize\bfseries}{\thesubsubsection}{1em}{}
\titleformat{\paragraph}[runin]
{\normalfont\normalsize\bfseries}{\theparagraph}{1em}{}
\titleformat{\subparagraph}[runin]
{\normalfont\normalsize\bfseries}{\thesubparagraph}{1em}{}

\titlespacing*{\chapter} {0pt}{50pt}{40pt}
\titlespacing*{\section} {0pt}{3.5ex plus 1ex minus .2ex}{2.3ex plus .2ex}
\titlespacing*{\subsection} {0pt}{3.25ex plus 1ex minus .2ex}{1.5ex plus .2ex}
\titlespacing*{\subsubsection}{0pt}{3.25ex plus 1ex minus .2ex}{1.5ex plus .2ex}
\titlespacing*{\paragraph} {0pt}{3.25ex plus 1ex minus .2ex}{1em}
\titlespacing*{\subparagraph} {\parindent}{3.25ex plus 1ex minus .2ex}{1em}

\newtheorem{theorem}{Theorem}[section]
\newtheorem{lemma}[theorem]{Lemma}
\newtheorem{proposition}[theorem]{Proposition}
\newtheorem{corollary}[theorem]{Corollary}

\theoremstyle{definition}
\newtheorem{definition}[theorem]{Definition}

\newtheorem{notation}[theorem]{Notation}
\newtheorem{example}[theorem]{Example}

\theoremstyle{remark}
\newtheorem{remark}[theorem]{Remark}

\DeclareMathOperator{\Diag}{Diag}

\DeclareMathOperator{\ide}{Id}

\newcommand{\mm}{\widetilde{M}}
\newcommand{\ov}{\overline}
\newcommand{\ot}{\otimes}

\newcommand{\hs}{\hspace{-0.5pt}}

\newcommand{\xcirc}{\hs \circ \hs}

\numberwithin{equation}{section}
\begin{document}

\title[On the classification of graded twisted planes]{On the classification of graded twisted planes}

\author{Ricardo Bances$^{1}$}
\email{rbances@pucp.edu.pe}

\author{Christian Valqui$^{1,}$$^2$}
\address{$^1$Pontificia Universidad Cat\'olica del Per\'u, Secci\'on Matem\'aticas, PUCP, Av. Universitaria 1801, San Miguel, Lima 32, Per\'u.}
\address{$^2$Instituto de Matem\'atica y Ciencias Afines (IMCA) Calle Los Bi\'ologos 245. Urb San C\'esar.
La Molina, Lima 12, Per\'u.}
\email{cvalqui@pucp.edu.pe}

\thanks{Christian Valqui was supported by PUCP-DGI-2019-1-0015.}
\thanks{Data sharing not applicable to this article as no datasets were generated or analysed during the current study.}
\thanks{Corresponding author: Christian Valqui}

\begin{abstract}
We use a representation of a graded twisted tensor product of $K[x]$ with $K[y]$ in $L(K^{\mathds{N}_0})$ in order
to obtain a nearly complete classification of these graded twisted tensor products via infinite matrices.
There is one particular example and
three main cases:  quadratic algebras classified in~\cite{CG2},
 a family called $A(n,d,a)$ with the $n+1$-extension property for $n\ge 2$, and a third case, not fully classified, which contains a
 family $B(a,L)$ parameterized by quasi-balanced sequences.
\end{abstract}

\maketitle

\tableofcontents

\section*{Introduction}
In~\cite{CSV} the authors introduced the notion of twisted tensor product of unital $K$-algebras, where $K$ is a unital ring.
We assume that $K$ is a field, and consider the
basic problem of classifying all twisted tensor products
of $A$ with $B$ for a given pair of algebras $A$ and $B$. In general this problem is out of reach, although some results have been obtained, mainly
for finite dimensional algebras (see~\cite{A}, \cite{AGGV},
 \cite{C}, \cite{GGV1}, \cite{JLNS} and~\cite{LN}).
In particular, in~\cite{GGV1} some families of twisted tensor products of $K[x]$ with $K[y]$ were found. The full classification of these tensor
products seems to be still out of reach, but
in~\cite{CG2} (see also~\cite{CG})  the graded twisted tensor products of $K[x]$ with $K[y]$ which yield quadratic algebras were completely
classified.
On the other hand the twisted tensor product of $K[x]/\langle x^n\rangle$ with an algebra $A$ can be represented in $M_n(A)$
(see~\cite{GGV2}*{Theorem 1.10}).
This representation can be generalized to finite dimensional algebras (see~\cite{A}).

In this article we start with a representation of a twisted tensor product of $K[x]$ with $K[y]$ in $L(K[x]^{\mathds{N}_0})$,
which is very similar to the representation in~\cite{GGV2}*{Theorem 1.10}. In the graded case this representation can be simplified further
to a representation of the graded twisted tensor product in the algebra $L(K^{\mathds{N}_0})$, embedded in the algebra of infinite matrices with entries in $K$. Thus we manage to translate the problem of classifying
all graded twisted tensor products of $K[x]$ with $K[y]$ into the problem of classifying infinite matrices with entries in $K$ satisfying certain
conditions (see Corollary~\ref{graded condition}).
With this method we show that all graded twisted tensor products of $K[x]$ with $K[y]$ can be classified into three main cases, except
 one particular example (see Proposition~\ref{caso particular}). The first case is the case of quadratic algebras, already classified in~\cite{CG2},
the second case yields a family called $A(n,d,a)$ and in the third case we have only  some partial classification results, and obtain a family called
$B(a,L)$.

One can describe a graded twisting tensor product of $K[x]$ with $K[y]$ by specifying how $y^k$ commutes with $x$, which means determining the
coefficients $a_i$'s in
\begin{equation}\label{commutation relations}
  y^kx=\sum_{i=0}^{k+1}a_i x^{k+1-i}y^{i}.
\end{equation}
For example, we obtain the quantum plane if the commutation relation is $yx=qxy$, and in Example~\ref{Ore} we explore the case $yx=bxy+ cy^2$.
\begin{definition}
  A graded twisting tensor product of $K[x]$ with $K[y]$ has the $n$-extension property, if the multiplicative structure
  of the algebra is determined completely by the commutation relations~\eqref{commutation relations}
  for $k=1,\dots,n-1$. For example, if  the relation $yx=ax^2+bxy+cy^2$ determines the multiplicative structure
  of the algebra, then the algebra has the $2$-extension property,
  and is called quadratic.
\end{definition}
In general,
the relation $yx=ax^2+bxy+cy^2$ yields a quadratic algebra, provided that $ac\ne 1$ and that $(b,ac)$ is not a root of any member of a certain family
of polynomials $Q_n(b,c)$. In this case our results match the results of~\cite{CG2}, which were obtained
with very different methods.

If $a\ne 0$, then such a
tensor product is equivalent to one with $a=1$, and so we will focus on the case $yx=x^2+bxy+cy^2$.
In the case $c=1$ one can show that necessarily $b=-1$ (see Lemma~\ref{lema clasifica}) and the resulting algebras are not quadratic, i.e., they do not have
 the 2-extension property.
We obtain a particular algebra with $y^kx=x^{k+1}-x^ky+y^{k+1}$ for all $k\in\mathds{N}$ (see Proposition~\ref{caso particular}).
This algebra does not have the $m$-extension property for any $m$.

For every graded twisted tensor product with $yx=x^2-xy+y^2$, which is not
the particular case mentioned above, there exists $n\ge 2$ such that
$$
  y^kx=x^{k+1}-x^ky+y^{k+1},\quad\text{for all $k<n$},\quad\text{and}\quad
  y^n x\ne x^{n+1}-x^ny+y^{n+1}.
$$
A central result is Proposition~\ref{clasificacion}, which shows that we have exactly two possibilities for $y^n x$. In the first case
$$
  y^nx=dx^{n+1}-dx^n y-a xy^n+(a+1)y^{n+1},
$$
for $a,d$ in $K$ satisfying certain conditions, namely, $(a,d)$ is not a root of any member of a certain family of polynomials $R_j(a,d)$.
 This yields a family
$A(n,d,a)$ of twisted tensor products which have the $(n+1)$-extension property (see sections~\ref{seccion 5} and~\ref{seccion 6}), which means that
the multiplication is determined by the commuting relations up to degree $n+1$.
The second case is treated in sections~\ref{seccion 7}, \ref{seccion familia} and~\ref{seccion 10},
and the commutation relation at degree $n+1$ is
$$
  y^nx=dx^{n+1}-x^n y+(a+1)y^{n+1},
$$
where $(a+1)d=1$.
Although the full classification is not achieved in this case, our methods show that one can achieve the classification of
all possible twisting maps up to any degree, with increasing amount of computational work.

 Moreover, we manage to find a family of twisted tensor products which we call $B(a,L)$, parameterized by
$a\in K\setminus\{0,-1\}$ and $L\in\mathcal{L}$, where $\mathcal{L}$ is the set of quasi-balanced sequences of positive integers (see
Definition~\ref{Def quasi balanced}). These sequences
are interesting on their own, for example they show a surprising connection to Euler's function $\varphi$. Every truncated quasi-balanced sequence
can be continued in several ways, which
implies that all members of the family $B(a,L)$ have not the $m$-extension property for any $m$.

The families that were found via the partial classification in the present article could be the smallest
examples of non-quadratic graded algebras. On one hand we have a family of algebras $B(a,L)$, which have not the $m$-extension property for any $m$. On the other hand, for any chosen $n\ge 2$ there is an algebra with the $n+1$-extension property in $A(n,d,a)$.
It would be very interesting to study the homological behaviour of these algebras.

The following table contains all possible graded twisting maps of $K[x]$ with $K[y]$. The only twisting maps that have not been fully classified are
in the last row.
For the families with $y x= x^2-xy+y^2$ we formulate the commuting relations
with respect to the infinite matrices $Y=\phi(y)$, $M=\phi(x)$ and $\mm=\phi(x-y)$, obtained from the faithful representation  (see
Remark~\ref{representation graded})
$$
  \phi: K[x]\otimes_\sigma K[y]\to L(K^{\mathds{N}_0}).
$$
For example
the relation $y x= x^2-xy+y^2$ corresponds to $Y\mm=M \mm$, and $Y^k \mm=M^{k}\mm$ stands for $y^kx=x^{k+1}-x^ky+y^{k+1}$.

\begin{table}[htb] CLASSIFICATION TABLE\vspace{0.5cm}\\
 \begin{tabular}{|c|c|c|l|c|}
  \hline
  Commutation relations & Classification \& & Reference & $m$-extension \\
     & Parameters &  & property \\ \hline
  $yx=bxy+cy^2$  & $b,c\in K$ &   Example~\ref{Ore} & Quadratic \\ \hline
  $yx=x^2+bxy+cy^2$  & $b,c\in K$ &   Theorem~\ref{clasificacion caso generico} & Quadratic \\
   & $Q_k(b,c)\ne 0, \forall k\in\mathds{N}$ &  & \\ \hline
  $Y^k \mm=M^{k}\mm$, $\forall k\ge 1$  & Particular example &   Proposition~\ref{caso particular} & No $m$-extension \\
  & &    &property for \\
  & &    &any $m$. \\ \hline
  $Y^k \mm=M^{k}\mm$, $\forall k<n$ & Family $A(n,d,a)$ &   Theorem~\ref{anda} and & $n+1$-extension\\
  $Y^n \mm=dM^n \mm-a \mm Y^{n}$ & $n\in\mathds{N}$, $a,d\in K$,  &   Corollary~\ref{Corolario anda} & property. \\
  & $R_k(a,d)\ne 0, \forall k\in\mathds{N}$ &   & \\  \hline
  $Y^k \mm=M^{k}\mm$, $\forall k<n$ & Family $B(a,L)$ &  Propositions~\ref{LL es necesario para BnL}   &  No $m$-extension \\
  $Y^k \mm=d^r M^{k}\mm$,&$a\in K\setminus \{-1,0\}$  &  and~\ref{LL es suficiente para BnL} &  property  for \\
  if $L_r< k<L_{r+1}$& $L\in\mathcal{L}$   &  & any $m$. \\
  $Y^k \mm=d^r M^{k+1}-d^{r-1}M^j Y$& Subfamily of &  &  \\
  $+a Y^{j+1}$,& case below  &  & \\
  if  $k=L_r$&  &    & \\ \hline
  $Y^k \mm=M^{k}\mm$, $\forall k<n$ & Not fully &   Section~\ref{seccion 7}  & Conjecture:\\
  $Y^n \mm=dM^{n+1}-M^nY$ & classified &   & No
  $m$-extension\\
  $+a Y^{n+1}$& &    &property. \\ \hline
\end{tabular}
\end{table}
\section{Preliminaries}\label{preliminares}
Let $K$ be a field and let $A$ and $B$ be unitary $K$-algebras. A  twisted tensor product of $A$ with $B$ over $K$ is an associative
algebra structure defined on
$A\ot B$,
such that the canonical maps $i_A\colon A\longrightarrow  A\ot_K B$ and $i_B\colon B\longrightarrow A\ot_K B$ are algebra maps satisfying
$a\ot b= i_A(a)i_B(b)$.
We will classify the graded twisted tensor products of $A=K[x]$ with $B=K[y]$. As is known (see e.g. \cite{CSV}) classifying the twisted
tensor products is equivalent to classifying the twisting maps $\tau: K[y]\otimes K[x]\to K[x]\otimes K[y]$, which are $K$-linear maps satisfying
\begin{enumerate}

  \item[(a)] $\tau(1\otimes a)=a\otimes 1$,

  \item[(b)] $\tau(y^r\ot 1)=1\ot y^r$,

  \item[(c)]$\tau(y^r\otimes ab)=(\mu_A\otimes C)\xcirc (A\otimes \tau)\xcirc (\tau\otimes A) (y^r\otimes a\otimes b)$,

  \item[(d)] $\tau(y^ry^t\otimes a)=(A\otimes\mu_C)\xcirc (\tau\otimes C)\xcirc (C\otimes \tau) (y^r\otimes y^t\otimes a)$.

\end{enumerate}
(See for example~\cite{CSV}*{Remark 2.4}). The multiplication on $A\otimes B$ is then defined by
\begin{equation}\label{multiplicacion en A ot B}
  \mu_{\tau}=(\mu_A\otimes \mu_B)\circ (A\otimes \tau \otimes B).
\end{equation}
By definition two twisting maps $\tau$ and $\tau'$ are isomorphic if and only if there are algebra automorphisms
$g: A\to A$ and $h: B\to B$ such that $\tau' = (g^{-1}\otimes h^{-1})\circ \tau \circ (h\otimes g)$.

Now, a linear map
$\tau: K[y]\otimes K[x]\to K[x]\otimes K[y]$ determines and is determined by linear maps $\gamma^r_j:K[x]\to K[x]$ for $r,j\in\mathds{N}_0$ such that
$\gamma^r_j(a)=0$
for fixed $r,a$ and sufficiently big $j$; via the formula
$$
  \tau(y^r\otimes a)=\sum_j \gamma^r_j(a)\otimes y^j.
$$
\begin{proposition}\label{proposicion condiciones twisting}
A linear map
$\tau: K[y]\otimes K[x]\to K[x]\otimes K[y]$ is a twisting map if and only if
\begin{enumerate}
  \item $\gamma^0_j=\delta_{j0}\ide$.
  \item $\gamma^r_j(1)=\delta_{jr}$.
  \item For all $r,j$ and all $a,b\in K[x]$,
    $$
      \gamma^r_j(ab)=\sum_{k=0}^{\infty}\gamma^r_k(a)\gamma^k_j(b).
    $$
    Note that for fixed $a,b\in K[x]$, the sum is finite.
  \item For all $r,j$ and $i<r$,
    $$
      \gamma^r_j=\sum_{l=0}^j\gamma^i_l\circ \gamma^{r-i}_{j-l}.
    $$
\end{enumerate}
\end{proposition}

\begin{proof}
A straightforward computation shows that these four conditions correspond to the four conditions~(a)--(d) that characterize a twisting map
(see for example~\cite{GGV1}*{Theorem 2.1}).
\end{proof}

If $\tau$ is a twisting map, then we will define a representation of the twisted  tensor product $K[x]\otimes_\tau K[y]$ on
$K[x]^{\mathds{N}_0}$ along the lines of~\cite{GGV2}*{Theorem 1.10}. For this note that the elements of $L(K[x]^{\mathds{N}_0})$ are the infinite
matrices with
entries in $K[x]$ indexed by $\mathds{N}_0\times \mathds{N}_0$ such that each row has only a finite number of non zero entries.
\begin{notation}\label{matrices Y y Z}
  Throughout this paper we denote by $Y$ and $Z$ the infinite matrices
  $$
    Y:=\begin{pmatrix}
     0 & 1 & 0 & 0 & \dots \\
     0 & 0 & 1 & 0 &  \\
     0 & 0 & 0 & 1 &  \\
     0 & 0 & 0 & 0&  \\
     \vdots &  &  &  & \ddots
    \end{pmatrix}\quad \text{ and }\quad
    Z:=\begin{pmatrix}
     0 & 0 & 0 & 0 & \dots \\
     1 & 0 & 0 & 0 &  \\
     0 & 1 & 0 & 0 &  \\
     0 & 0 & 1 & 0 &  \\
     \vdots &  &  &  & \ddots
    \end{pmatrix}
  $$
  in $L(K[x]^{\mathds{N}_0})$. Note that $YZ=\ide$ and $ZY=\ide-E_{00}$.
\end{notation}
\noindent For any infinite matrix $B$ we have $(Y^k B)_{ij}=B_{i+k,j}$ and $(B Y^k)_{i,j}=\begin{cases}
                                                        B_{i,j-k}, & \mbox{if } j\ge k \\
                                                        0, & \mbox{otherwise}.
                                                      \end{cases}$

\noindent If now $\tau: K[y]\otimes K[x]\to K[x]\otimes K[y]$ is a twisting map determined by the $K$-linear maps $\gamma^i_j$,  for each $a\in K[x]$
we define the infinite matrix
$M(a)\in  L(K[x]^{\mathds{N}_0})$ by $M(a)_{ij}=\gamma_j^i(a)$. This matrix satisfies the finiteness condition, since
$$
  \tau(y^i\otimes a)=\sum_j \gamma^i_j(a)\otimes y^j\in K[x]\otimes K[y],
$$
so $\gamma_j^i(a)\ne 0$ only for a finite number of $j$'s.

\begin{remark} \label{conditions for M(a)}
  By conditions~(2) and~(3) we have $M(1)=\ide$ and $M(ab)=M(a) M(b)$ for all $a,b\in K[x]$.
\end{remark}

\begin{proposition}\label{representation}
Let $\tau: K[y]\otimes K[x]\to K[x]\otimes K[y]$ be a twisting map.
  The formulas $\psi(a\otimes 1)=M(a)$ and $\psi(1\otimes y)=Y$ determine an injective algebra map (faithful representation)
  $\psi: K[x]\otimes_{\tau}K[y]\to
  L(K[x]^{\mathds{N}_0})$.
\end{proposition}

\begin{proof}
  By Remark~\ref{conditions for M(a)} we have $\psi(1)=1$. For $\psi$ to be an algebra map, we need to define
  $$
  \psi(1\otimes y^k)=Y^k\quad\text{and}\quad \psi(a\otimes y^k)=\psi((a\otimes 1)(1\otimes y^k))=M(a)Y^k.
  $$
  But then $\psi$ is compatible with the multiplication of elements of the form $ (1\otimes y^i)(1\otimes y^k)$, and by
  Remark~\ref{conditions for M(a)} it is also compatible with elements of the form $(a\otimes 1)(b\otimes 1)$.
  Since by~\eqref{multiplicacion en A ot B} we know that
  $$
  (1\otimes y^k)(a\otimes 1)=(\mu_{K[x]}\otimes \mu_{K[y]})(1\otimes \tau(y^k\otimes a)\otimes 1)
  =\sum_{u} \gamma_u^k(a)\otimes y^u,
  $$
  we have to prove that
  $$
    \psi(1\otimes y^k)\psi(a\otimes 1)=\sum_{u} \psi(\gamma_u^k(a)\otimes 1)\psi(1\otimes y^u)
  $$
  for $a\in K[x]$. Using condition~(4) of Proposition~\ref{proposicion condiciones twisting} we obtain
  \begin{eqnarray*}
    (\psi(1\otimes y^k)\psi(a\otimes 1))_{ij}   &=&  (Y^k M(a))_{ij}=\gamma_j^{i+k}(a)=\sum_{u=0}^{j}\gamma_{j-u}^{i}(\gamma_u^k(a))\\
    &=& \sum_{u=0}^j M(\gamma_u^k(a))_{i,j-u}=\sum_{u=0}^{\infty} (M(\gamma_u^k(a))Y^{u})_{ij}\\
    &=& \sum_{u} \left(\psi(\gamma_u^k(a)\otimes 1)\psi(1\otimes y^u) \right)_{ij},
  \end{eqnarray*}
  which concludes the proof that $\psi$ is an algebra map.
  The injectivity follows from the fact that the composition of $\psi$ with the surjection onto the first row gives the canonical linear isomorphism
  $$
    K[x]\otimes_{\tau}K[y]\xrightarrow{\cong}  \bigoplus_{i\in \mathds{N}_0} K_i,\quad \text{where $K_i\cong K[x]$.}
  $$
\end{proof}

\begin{remark}
  The previous representation can be related to the right regular representation of the algebra $K[x]\otimes_{\tau}K[y]$. In fact, we can write
  $K[x]\otimes_{\tau}K[y]$ as a right module over itself, as
  $$
  K[x]\otimes_{\tau}K[y]\cong \sum_{n=0}^{\infty} K_i,\quad\text{with $K_i\cong K[x]$,}
  $$
  so we are considering infinite column vectors with entries in $K[x]$. Then the multiplication by $y$ and $x$ from the right are
  represented by the multiplication from the left by the matrices
  $$
  \begin{pmatrix}
     0 & 0 & 0 & 0 & \dots \\
     1 & 0 & 0 & 0 &  \\
     0 & 1 & 0 & 0 &  \\
     0 & 0 & 1 & 0 &  \\
     \vdots &  &  &  & \ddots
    \end{pmatrix}
    \quad\text{and}\quad
  \begin{pmatrix}
     x & \gamma_0^1(x) & \gamma_0^2(x) & \gamma_0^3(x) & \dots \\
     0 & \gamma_1^1(x) & \gamma_1^2(x) & \gamma_1^3(x) &  \\
     0 & \gamma_2^1(x) & \gamma_2^2(x) & \gamma_2^3(x) &  \\
     0 & \gamma_3^1(x) & \gamma_3^2(x) & \gamma_3^3(x) &  \\
     \vdots &  &  &  & \ddots
    \end{pmatrix}.
  $$
  This gives a representation of the opposite algebra $(K[x]\otimes_{\tau}K[y])^{Op}$ in the algebra of infinite
  matrices with only finitely many non-zero
  entries in each column. If we take the transposed matrices, then we obtain the representation in Proposition~\ref{representation}.
\end{remark}

\begin{notation}
Let $M\in L(K[x]^{\mathds{N}_0})$ and for fixed $k,j$ write $M_{kj}=a_0+a_1 x+\dots +a_n x^n\in K[x]$.
Then we will evaluate this polynomial at $M$ setting
  $$
  M_{kj}(M)=a_0 \ide+a_1 M+\dots +a_n M^n\in L(K[x]^{\mathds{N}_0}).
  $$
\end{notation}

\begin{proposition}\label{condiciones para twisting}
  Let $M\in L(K[x]^{\mathds{N}_0})$ be such that $M_{0j}=x\delta_{0j}$ and
  \begin{equation}\label{condiciones twisting}
    Y^k M=\sum_{j\ge 0} M_{kj}(M) Y^j.
  \end{equation}
  (Note that the sum is finite).
  Then the maps $\gamma^r_j$ defined by $\gamma^r_j(x^i)\coloneqq (M^i)_{kj}\in K[x]$ determine a twisting map.
\end{proposition}

\begin{proof}
  We will prove that $\gamma^r_j$ satisfies~(1) to~(4) of Proposition~\ref{proposicion condiciones twisting}.

  (1): This is clear since $(M^i)_{0j}=x^i \delta_{0j}$.

  (2): Follows from $M^0=\ide$.

  (3): Clearly
  $$
    \gamma^r_j(x^{i+l})=(M^{i+l})_{rj}=\sum_{k=0}^{\infty}(M^{i})_{rk}(M^{l})_{kj}=\sum_{k=0}^{\infty}\gamma^r_k(x^i)\gamma^k_j(x^l),
  $$
  as desired.

  (4): For all $i,j,k,l$ we have to prove
  \begin{equation}\label{composition law}
    \gamma^{l+k}_j(x^i)=\sum_{u=0}^{j} \gamma^l_u (\gamma^{k}_{j-u}(x^{i})
  \end{equation}
   by induction in $i$. For $i=1$ we have
  \begin{eqnarray*}
     \gamma^{l+k}_j(x) &=& M_{k+l,j} \\
     &=& (Y^k M)_{lj}\\
     &=&\sum_{u\ge 0}(M_{ku}(M) Y^u)_{lj}\\
     &=&\sum_{u= 0}^j(M_{ku}(M))_{l,j-u}\\
     &=&\sum_{u= 0}^j\gamma^{l}_{j-u}(M_{ku})\\
     &=& \sum_{u=0}^j\gamma^{l}_{j-u}(\gamma^k_u(x))\\
     &=& \sum_{u=0}^{j} \gamma^l_u (\gamma^{k}_{j-u}(x)),
  \end{eqnarray*}
  where the fifth equality follows from the fact that for $M_{il}(x)=\sum_s a_s x^s$ we have
  $$
    (M_{il}(M))_{rj}=\sum_s a_s(M^s)_{rj}=\sum_s a_s\gamma^r_j(x^s)= \gamma^r_j\left(\sum_s a_s x^s\right)=\gamma^r_j(M_{il}).
  $$
  Assume~\eqref{composition law} is valid for $i$. Then
  \begin{eqnarray*}
    \gamma^{l+k}_j(x^{i+1}) &=& \sum_{r} \gamma^{l+k}_r(x^i)\gamma_{j}^{r}(x) \\
    &=& \sum_{r} \sum_{s=0}^r \gamma^l_s (\gamma^{k}_{r-s}(x^i))\gamma_{j}^{r}(x)\\
    &=&\sum_{s\ge 0} \sum_{n\ge 0} \gamma^l_s (\gamma^{k}_{n}(x^i))\gamma_{j}^{s+n}(x)\\
    &=&\sum_{n,s\ge 0} \gamma^l_s (\gamma^{k}_{n}(x^i))\sum_{u=0}^{j}\gamma_u^s(\gamma_{j-u}^{n}(x))\\
    &=&\sum_{n\ge 0}\sum_{u=0}^{j}\sum_{s\ge 0} \gamma^l_s (\gamma^{k}_{n}(x^i))\gamma_u^s(\gamma_{j-u}^{n}(x))\\
    &=&\sum_{n\ge 0}\sum_{u=0}^{j} \gamma^l_u (\gamma^{k}_{n}(x^i)\gamma_{j-u}^{n}(x))\\
    &=&\sum_{u=0}^{j} \gamma^l_u \left(\sum_{n\ge 0}\gamma^{k}_{n}(x^i)\gamma_{j-u}^{n}(x)\right)\\
    &=&\sum_{u=0}^{j} \gamma^l_u (\gamma^{k}_{j-u}(x^{i+1})),
  \end{eqnarray*}
  as desired.
\end{proof}

Now we assume that the potential twisting map is graded, that means that $\tau(y^r\otimes x^i)=\sum_{j=0}^{i+r} a_j x^{i+r-j}\otimes y^j$, and so
the maps $\gamma^r_j$ are homogeneous of degree $r-j$.

\begin{corollary}\label{graded condition}
  Let $M\in L(K^{\mathds{N}_0})$ be such that $M_{0j}=\delta_{0j}$, $M_{kj}=0$ for $j>k+1$
  and
  \begin{equation}\label{fundamental matrix equality}
    Y^k M=\sum_{j= 0}^{k+1} M_{kj}M^{k+1-j} Y^j.
  \end{equation}
  Then the maps $\gamma^r_j$ defined by
  \begin{equation}\label{gamma con M}
  \gamma^r_j(x^i)\coloneqq (M^i)_{kj}x^{k+i-j}
  \end{equation}
  determine a graded twisting map.
  Conversely, if $\tau:K[y]\otimes K[x]\to K[x]\otimes K[y]$ is a graded twisting map, then the corresponding
  $\gamma^r_j$'s determine via~\eqref{gamma con M} a matrix $M$ satisfying~\eqref{fundamental matrix equality},
  $M_{0j}=\delta_{0j}$ and $M_{kj}=0$ for $j>k+1$.
\end{corollary}
\begin{proof}
  The map $\tau$ determined by the $\gamma^r_j$'s is clearly graded, so we only need to show that the matrix $\widetilde{M}$ defined by
  $\widetilde{M}_{kj}=M_{kj}x^{k+1-j}$ satisfies~\eqref{condiciones twisting}, since
  $\gamma^r_j(x^i)= (\widetilde{M}^i)_{kj}$ and $\mm_{0j}= x \delta_{0j}$. We obtain
  \begin{eqnarray*}
    (Y^k\widetilde{M})_{rs} &=& (Y^k M)_{rs} x^{r+k+1-s}= x^{r+k+1-s}\sum_{j=0}^{k+1} M_{kj}(M^{k+1-j}Y^j)_{rs} \\
    &=&  x^{r+k+1-s} \sum_{j=0}^{\max\{k+1,s\}} M_{kj}(M^{k+1-j})_{r,s-j} \\
    &=& \sum_{j=0}^{\max\{k+1,s\}} (M_{kj}\widetilde{M}^{k+1-j})_{r,s-j}\\
    &=&\sum_{j=0}^{k+1}  (\widetilde{M}_{kj}(\widetilde{M})Y^j)_{rs}\\
    &=&\sum_{j\ge 0}  (\widetilde{M}_{kj}(\widetilde{M})Y^j)_{rs},
  \end{eqnarray*}
  where we use that $\widetilde{M}_{kj}(\widetilde{M})=M_{kj}\widetilde{M}^{k+1-j}$.
  Now the result follows from Proposition~\ref{condiciones para twisting}, and the converse is straightforward.
\end{proof}

\begin{remark}\label{representation graded}
  In the graded case the formulas $\phi(x)=M$ and $\phi(y)=Y$ define an injective algebra map (faithful representation)
  $\phi: K[x]\otimes_{\tau}K[y]\to L(K^{\mathds{N}_0})$.
\end{remark}

\section{Construction of the matrices associated with a twisting map}
  In order to classify the graded twisting maps, we have to classify the matrices $M$ satisfying the conditions of
  Corollary~\ref{graded condition}. Note that
  $M=\psi(x)|_{x=1}=M(x)|_{x=1}$ in the notation of Proposition~\ref{representation}.
  We will write $M_{10}=a$, $M_{11}=b$ and $M_{12}=c$ for such a matrix. In some cases the values of $a$, $b$ and $c$ determine completely the matrix
  $M$ (and hence the twisting map).

\begin{example}\label{Ore}
  If $a=0$, then the equality $yx=bxy+cy^2$ implies
  $$
    y^2x=y(bxy+cy^2)=b(yx)y+cy^3=b(bxy+cy^2)y+cy^3=b^2 xy^2+c(b+1)y^3,
  $$
  and a straightforward inductive argument shows that
  $$
    y^k x=b^k xy^k+c[k]_by^{k+1},
  $$
  where $[k]_b$ denotes the $q$-number defined by $[k]_b=1+b+b^2+\dots+b^{k-1}$.
  So the only possible non zero entries of $M$ are $M_{nn}=b^n$ and $M_{n,n+1}=c[n]_b$ and the corresponding matrix is
  $$
    \begin{pmatrix}
        1 & 0 & 0 & 0 & 0 & \dots  \\
        0 & b & c & 0 & 0 &   \\
        0 & 0 & b^2 & c(b+1) & 0 &   \\
        0 & 0 & 0 & b^3 & c(b^2+b+1) &   \\
        \vdots & \vdots &  &  0 & b^4 & \ddots \\
        \vdots &  &  & \vdots &  0 & \ddots
    \end{pmatrix},
  $$
  and the twisting map is given by
  \begin{align*}
    \tau(1\otimes x)& = x\otimes 1  ,\\
    \tau(y\otimes x) & = b x\otimes y+c(1\otimes y^2),\\
    \tau(y^2\otimes x) & = b^2 x\otimes y^2+c(b+1)(1\otimes y^3),\\
    \tau(y^3\otimes x) & = b^3 x\otimes y^3+c(b^2+b+1)(1\otimes y^4),\\
    &\vdots.
  \end{align*}
  By Corollary~\ref{graded condition} the matrix equalities
  \begin{align*}
    Y M & = b MY+c Y^2,\\
    Y^2 M & = b^2 M Y^2+c(b+1) Y^3,\\
    Y^3 M & = b^3 M Y^3+c(b^2+b+1)Y^4,\\
    &\vdots
  \end{align*}
  guarantee that $\tau$ is a twisting map. In order to prove these equalities, one notes first that the first equality implies all the others (use
  induction). Then we check directly that $Y M  = b MY+c Y^2$:
  $$
    \begin{pmatrix}
        0 & b & c & 0  &   \\
        0 & 0 & b^2 & c(b+1)  &   \\
        0 & 0 & 0 & b^3  &   \\
        \vdots & \vdots &  &  0  & \ddots \\
        \vdots &  &  & \vdots  & \ddots
    \end{pmatrix}
    =b \begin{pmatrix}
        0&1 & 0 & 0 & 0  & \dots  \\
        0 &0& b & c & 0  &   \\
        0 &0 &0 & b^2 & c(b+1)  &   \\
        \vdots && \vdots &  &  0  & \ddots \\
        \vdots &&  &  & \vdots  & \ddots
    \end{pmatrix}
    + c\begin{pmatrix}
        0&0 & 1 & 0 & 0  & \dots  \\
        0 &0& 0 & 1 & 0  &   \\
        0 &0 &0 & 0 & 1  &   \\
        \vdots && \vdots &  &  0  & \ddots \\
        \vdots &&  &  & \vdots  & \ddots
    \end{pmatrix}.
  $$
\end{example}

\begin{lemma}\label{tecnico}
  Let $A$ be an associative $K$-algebra, $k\ge 2$, $x,y\in A$, $M_{ij}\in K$ for $0\le j<k$ and $0\le i\le j+1$ such that $M_{0j}=\delta_{j0}$,
  $(M_{10},M_{11},M_{12})=(1,b,c)$ and
  $$
    y^j x=\sum_{i=0}^{j+1} M_{ji}x^{j+1-i}y^i, \quad\text{for $0\le j<k$.}
  $$
 Then
  \begin{equation}\label{induccion}
    (1-M_{k-1,k})y^k x=\sum_{s=0}^{k+1}\ov M_{ks} x^{k+1-s} y^s,
  \end{equation}
  where
  $$
    \ov M_{k0}=\sum_{i=0}^{k-1}M_{k-1,i}M_{i0},\quad \ov M_{k,k+1}=c+b M_{k-1,k}\quad\text{and}\quad
    \ov M_{ks}=bM_{k-1,s-1}+\sum_{i=s-1}^{k-1}M_{k-1,i}M_{is},
  $$
  for $s=1,\dots,k$.
\end{lemma}

\begin{proof}
  We have
  \begin{eqnarray*}
    y^k x &=& y^{k-1}yx=y^{k-1}(x^2+bxy+cy^2) \\
     &=& \left(\sum_{i=0}^k M_{k-1,i} x^{k-i}y^i\right) x+ b \sum_{i=0}^k M_{k-1,i} x^{k-i}y^{i+1}+c y^{k+1}  \\
     &=& M_{k-1,k}y^k x+ \sum_{i=0}^{k-1} M_{k-1,i} x^{k-i}(y^i x)+ b \sum_{s=1}^k M_{k-1,s-1} x^{k+1-s}y^s+ (c+b M_{k-1,k})y^{k+1}.
  \end{eqnarray*}
  Since
  \begin{eqnarray*}
    \sum_{i=0}^{k-1} M_{k-1,i} x^{k-i}(y^i x) &=&\sum_{i=0}^{k-1} M_{k-1,i} x^{k-i}\left(\sum_{s=0}^{i+1} M_{is}x^{i+1-s}y^s\right) \\
     &=&\sum_{i=0}^{k-1}\sum_{s=1}^{i+1} M_{k-1,i} M_{is}x^{k+1-s}y^s +\sum_{i=0}^{k-1} M_{k-1,i} M_{i0} x^{k+1} \\
     &=&\sum_{s=1}^{k} \left(\sum_{i=s-1}^{k-1} M_{k-1,i} M_{is}\right) x^{k+1-s}y^s +\sum_{i=0}^{k-1} M_{k-1,i} M_{i0} x^{k+1},
  \end{eqnarray*}
  the result follows.
\end{proof}

\begin{remark}\label{unicidad en lema}
  Let $M\in L(K^{\mathds{N}_0})$ be such that $M_{0j}=\delta_{0j}$, $M_{ji}=0$ for $i>j+1$
  and for some $k>1$,
  \begin{equation}\label{fundamental matrix equality1}
    Y^j M=\sum_{i= 0}^{j+1} M_{ji}M^{j+1-i} Y^i,\quad \text{for $j<k$.}
  \end{equation}
  Then $x=M$ and $y=Y$ satisfy the assumptions of the lemma. The equality~\eqref{induccion} reads
  $$
   (1-M_{k-1,k})Y^{k} M=\sum_{s=0}^{k+1}\ov M_{ks} M^{k+1-s} Y^s,
  $$
  and if we take the entry $(0,i)$, then the left hand side gives
  $$
  \left((1-M_{k-1,k})Y^{k} M\right)_{0i}=(1-M_{k-1,k})M_{ki},
  $$
  and the right hand gives
  $ \ov M_{ki}$, since
  $$
  \left(M^{k+1-s} Y^s\right)_{0i}=\left(M^{k+1-s}\right)_{0,i-s}= \delta_{is}.
  $$
  Hence in a twisting map with $M_{k-1,k}\ne 1$, the coefficients $M_{ki}$ are determined uniquely by the coefficients
  $M_{ji}$ with $j<k$.
\end{remark}

\begin{remark}\label{a distinto de cero}
  Given a twisting map $\tau$ such that $a\ne 0$,
  we can replace $\tau$ by the isomorphic twisting map
  $\tau'=(f^{-1}\otimes \ide)\circ \tau \circ (\ide\otimes f)$ where $f(x)=ax$. Then for $\tau'$ we have $a'=1$, $b'=b$ and $c'=ca$.
  So we can and will assume that $a=1$.
\end{remark}

\begin{proposition}\label{bc determinan M}
  Assume $M$ determines a graded twisting map.
  Assume $a=1$ and $M_{k,k+1}\ne 1$ for all $k\ge 1$. Then $b$ and $c$ determine uniquely the matrix $M$ (and hence the twisting map).
\end{proposition}
\begin{proof}
  By Remark~\ref{unicidad en lema} in that case the entries $M_{ji}$ with $j<k$ determine uniquely the entries $M_{ki}$, hence, by induction,
   the entries $M_{1j}$, i.e.,  $M_{10}=1$, $M_{1,1}=b$, and $M_{12}=c$, determine the whole matrix.
\end{proof}

However not every choice of $b$ and $c$ is valid, as we will see.
\begin{lemma}\label{lema clasifica}
Let $M$ be a matrix of a graded twisting map and set $b_n=M_{nn}$ and $c_n=M_{n,n+1}$. Then
  \begin{equation}\label{formula para c}
    c_{n+1}(1-c_n)=b c_n+c,\quad\text{for all $n$.}
  \end{equation}
  Moreover, if $c\ne 1$, then $c_n\ne 1$ for all $n$, and if $c=1$, then $b=-1$ and $(1-c_n)(1-c_{n+1})=0$ for all $n$.
\end{lemma}

\begin{proof}
  We have
  $$
  (YM)_{n,n+2}=M_{n+1,n+2}=c_{n+1},\quad (MY)_{n,n+2}=M_{n,n+1}=c_n,\quad (Y^2)_{n,n+2}=1
  $$
  and
  $$
   (M^2)_{n,n+2}=\sum_k M_{nk}M_{k,n+2}=M_{n,n+1}M_{n+1,n+2}=c_n c_{n+1},
  $$
  since $M_{nk}=0$ for $k>n+1$ and $M_{k,n+2}=0$ for $k<n+1$.
  The entry at $(n,n+2)$ of the matrix equality
  $$
  YM=M^2+b MY+c Y^2,
  $$
  which holds by Corollary~\ref{condiciones para twisting}, gives
   $c_{n+1}=c_n c_{n+1}+b c_n +c$, from which~\eqref{formula para c} follows.

  Assume that $c\ne 1$ and assume by contradiction
  that $c_n=1$ for some $n>1$.
  Then from~\eqref{formula para c} for $n$ we obtain $b=-c$. Equation~\eqref{formula para c} again implies that
   $c_2=c$, $c_3=c_2$, and so on, contradicting $c_n\ne 1$.

  Finally, assume that $c=1$, then from~\eqref{formula para c} for $n=1$ we obtain $b=-1$ and so~\eqref{formula para c} reads $c_{n+1}(1-c_n)=-c_n+1$,
  hence $(1-c_n)(1-c_{n+1})=0$, as desired.
\end{proof}

If $c\ne 1$, then not all values of $b$ and $c$ yield twisting maps. For example, if $c\ne 1$ and $1-c=c(b+1)$, then by~\eqref{formula para c}
we would have $c_2=\frac{c(b+1)}{1-c}=1$, which is impossible by the previous lemma.

When $c\ne 1$, the first formulas for $c_n$ are
$$
  c_2=\frac{c(b+1)}{1-c},\quad c_3=\frac{c(1+b+b^2-c)}{1-2c-bc},\quad c_4=\frac{c(1+b)(1+b^2-2c)}{1-(3+2b+b^2)+c^2},
$$
and in general we have
$$
  c_n=\frac{cP_n}{Q_n},
$$
where $P_n$ and $Q_n$ are polynomials in $b$ and $c$. Moreover, the formula~\eqref{formula para c} yields the
recursive rules
\begin{equation}\label{recursion}
  P_{n+1}=b P_n+Q_n\quad\text{and}\quad Q_{n+1}=Q_n-cP_n.
\end{equation}

Given $b$ and $c$, these values are defined even when some $c_n=1$, and that happens if and only if $Q_{n+1}=Q_{n+1}(b,c)=0$.

\begin{corollary}\label{b y c determinan M}
  Let $K$ be a field and let $b,c\in K$ with $c\ne 1$. If $b$ and $c$ determine a (necessarily unique)
  twisting map via Proposition~\ref{bc determinan M},
  then $Q_n(b,c)\ne 0$ for all $n\in\mathds{N}$, where the polynomials $P_n,Q_n\in K[b,c]$ are defined by $P_1=1$,
  $Q_1=1$ and the recursive rules~\eqref{recursion}.
\end{corollary}

\begin{proof}
  By Proposition~\ref{bc determinan M} and the previous discussion.
\end{proof}

In order to prove the converse of Corollary~\ref{b y c determinan M}, we consider the valuation on the algebra $L(K[x]^{\mathds{N}_0})$
given by
$$
  w(M):=\inf\{i-j, m_{ij}\ne 0\}
$$
for $M=(m_{ij})_{i,j\in\mathds{N}_0}$, and we also set $w(0)=+\infty$. For example $w(E_{10})=1$ and $w(E_{01})=-1=w(E_{10}+E_{01})$. Note that for
some $M$ we can have $w(M)=-\infty$.

\begin{proposition}\label{proposition valuacion}
  Let $M,N\in L(K[x]^{\mathds{N}_0})$. Then
  \begin{enumerate}
    \item $w(M+N)\ge \min\{w(M),w(N)\}$.
    \item If $w(M)\ne w(N)$, then $w(M+N)= \min\{w(M),w(N)\}$.
    \item $w(MN)\ge w(M)+w(N)$.
  \end{enumerate}
\end{proposition}
\begin{proof}
  Straightforward.
\end{proof}

\begin{definition}
  We say that $M$ is homogeneous if $m_{ij}=0$ when $i-j\ne w(M)$, and we denote by $M^{(k)}$ the homogenous component of $M$ of weight $k$ given by
  $(M^{(k)})_{ij}=\delta_{i-j,k}M_{ij}$.
\end{definition}

For example, consider the matrices $Y$ and $Z$ of Notation~\ref{matrices Y y Z}, given by $Y_{ij}=\delta_{i+1,j}$ and $Z_{ij}=\delta_{i,j+1}$. Then
both are homogeneous
with $w(Y)=-1$ and $w(Z)=1$.

Consider the subalgebra  $\mathcal{R}\subset L(K[x]^{\mathds{N}_0})$ consisting of the homogeneous matrices of weight zero. If $M$ is homogeneous of
weight $k>0$, then
the matrix $N=(N_{ij})\in \mathcal{R}\subset L(K[x]^{\mathds{N}_0})$ given by
$N_{ij}=\delta_{ij} M_{k+i,i}$,
satisfies
\begin{equation}\label{parte homogenea positiva}
  M=Z^k N.
\end{equation}

On the other hand, if $M$ is homogeneous of weight $k<0$, then
the matrix $N=(N_{ij})\in \mathcal{R}\subset L(K[x]^{\mathds{N}_0})$ given by
$N_{ij}=\delta_{ij} M_{i,i-k}$
satisfies
\begin{equation}\label{parte homogenea negativa}
  M=N Y^k.
\end{equation}

It follows that for $M$ with $w(M)>-\infty$ we have a decomposition
\begin{equation}\label{descomposicion}
  M=\sum_{j=w(M)}^{0} N_j Y^j +\sum_{k>0} Z^k N_k
\end{equation}
for some $N_k\in \mathcal{R}$, where the infinite sum converges in the $Z$-adic topology. Note that if $w(M)>0$, then the first sum is empty.

We define the shift operator and its left inverse on
$\mathcal{R}$ by setting
$$
  SA:=\Diag(0,a_0,a_1,a_2,\dots)\quad\text{and}\quad TA:=\Diag(a_1,a_2,\dots),
$$
for $A=\Diag(a_0,a_1,a_2,\dots)\in\mathcal{R}$.
Note that for a matrix $A\in \mathcal{R}$ we have $ZAY= SA$ and $YAZ= TA$.

\begin{proposition} \label{M cumple ecuacion en grado dos}
  Let $K$ be a field and let $b,c\in K$ with $Q_n(b,c)\ne 0$ for all $n\in\mathds{N}$.  Then there exists a unique matrix
  $M\in L(K[x]^{\mathds{N}_0})$ with $w(M)\ge -1$, such that
  \begin{equation}\label{ecuacion en grado 2 para M para b y c}
    YM=M^2+bMY+cY^2.
  \end{equation}
\end{proposition}

\begin{proof}
We will construct homogeneous components $M^{(j)}$ for $j\ge -1$, such that $M=\sum_{j\ge -1}M^{(j)}$
satisfies~\eqref{ecuacion en grado 2 para M para b y c}.
 Note that the equality~\eqref{ecuacion en grado 2 para M para b y c}
  is true if and only if it holds for the homogeneous components of weight $k$ for all $k$, i.e. if
  \begin{equation}\label{ecuacion graduada}
    (YM)_{k}=(M^2+b MY+cY^2)_k
  \end{equation}
  for all $k\ge -2$.  We will construct recursively $M^{(-1)}$, $M^{(0)}$, $M^{(1)}$, $M^{(2)}$,\dots, $M^{(j)}$, such that~\eqref{ecuacion graduada}
  holds for $k=-2,-1,0,1,\dots,j-1$. This yields an inductive construction of  the unique $M$ such
  that~\eqref{ecuacion en grado 2 para M para b y c}
  holds.

We write $M^{(-1)}=CY$ and $M^{(j)}=Z^j B_j$ for $j\ge 0$, for some diagonal matrices $C,B_j\in\mathcal{R}$, and so
$$
M=\sum_{j\ge -1} M^{(j)}= CY+B_0+\sum_{j\ge 1}Z^j B_j.
$$

  Note that
  $$
    (YM)_{-2}=YCY=(SC)Y^2,\quad (YM)_{-1}=YB_0=(SB_0)Y,\quad \text{and}\quad (YM)_{j}=Z^{j}B_{j+1}
  $$
  for $j\ge 0$.
  Note also that
  $$
    (MY)_{-2}=CY^2,\quad (MY)_{-1}=B_0 Y,\quad \text{and}\quad (MY)_{j}=Z^{j+1}B_{j+1}Y=Z^j (TB_{j+1})
  $$
  for $j\ge 0$.

  Finally,
  $$
    (M^2)_{-2}=CYCY=C(SC)Y^2,\quad (M^2)_{-1}=(CY)B_0+B_0CY= C(SB_0)Y+B_0CY,
  $$
  and
  \begin{eqnarray*}
    (M^2)_{j}&=&CY Z^{j+1}B_{j+1}+Z^{j+1}B_{j+1}CY+ \sum_{i=0}^{j} Z^i B_i Z^{j-i}B_{j-i}\\
    &=&  Z^{j}(S^j C)B_{j+1}+Z^{j} T(B_{j+1}C)+ Z^{j}\sum_{i=0}^{j} (S^{j-i} B_i) B_{j-i}
  \end{eqnarray*}
  for $j\ge 0$.

  For $k=-2$ the equality~\eqref{ecuacion graduada} reads
  $$
    (SC)Y^2=C(SC)Y^2+bC Y^2+c Y^2,
  $$
  and since multiplying by $Y$ on the right is injective, we have
  $$
    (SC)=C(SC)+bC+c \mathds{1},
  $$
  Hence $(\mathds{1}-C)SC=bC+c\mathds{1}$, and the $n$th entry reads $(1-c_n)c_{n+1}=bc_n+c$ which is equality~\eqref{formula para c}.
  Thus we can construct recursively $c_{n+1}$, since $Q_n(b,c)\ne 0$ guarantees that $c_n\ne 1$ for all $n$.  This proves that $b$ and $c$ determine
  uniquely $C=M_{-1}$ such that~\eqref{ecuacion graduada} holds for $k=-2$.

  For $k=-1$ the equality~\eqref{ecuacion graduada} reads
  $$
    (SB_0)Y=C(SB_0)Y+B_0CY+bB_0Y,
  $$
  and since multiplying by $Y$ on the right is injective, we have
  $$
    (SB_0)=C(SB_0)+B_0C+bB_0.
  $$
  So we have a recursive formula for $B_0$:
  $$
    (B_0)_{n+1}(1-c_n)=(B_0)_n(c_n+b).
  $$
  Since $(1-c_n)\ne 0$ and we already have $(B_0)_0=1$ and $(B_0)_1=b$, this formula determines a unique $B_0$ such that the
  equality~\eqref{ecuacion graduada} for $k=-1$ is satisfied.

  For $j\ge 0$ the equality~\eqref{ecuacion graduada} reads
  $$
    Z^{j}B_{j+1} =  Z^{j}(S^j C)B_{j+1}+Z^{j} T(B_{j+1}C)+ Z^{j}\sum_{i=0}^{j} (S^{j-i} B_i) B_{j-i}+bZ^j (TB_{j+1}),
  $$
  and since multiplication by $Z^j$ at the left is injective we have
  $$
    (\mathds{1} -  (S^j C))B_{j+1}=T(B_{j+1})(TC+b\mathds{1})+ \sum_{i=0}^{j} (S^{j-i} B_i) B_{j-i}.
  $$
  Assume we have constructed inductively $C$ and $B_i$ for $i=0,\dots,j$ such that~\eqref{ecuacion graduada} is satisfied for
  $k=-2,\dots, j-1$. Then set $R:=\sum_{i=0}^{j} (S^{j-i} B_i) B_{j-i}$, which depends only on $B_i$ for $i=0,\dots,j$,
  and we obtain a recursive formula
  $$
    (B_{j+1})_{n}(1-c_{n+j})=(B_{j+1})_{n-1}(c_{n-1}+b)+R_n,
  $$
  which yields a unique $B_{j+1}$ such that~\eqref{ecuacion graduada} is satisfied for
  $k=-2,\dots, j$. Note that the formula is valid for $n=0$ setting $(B_{j+1})_{-1}=c_{-1}=0$.

  This proves that there is a unique
  $$
    M=CY+B_0+\sum_{j\ge 1}Z^j B_j,
  $$
  satisfying~\eqref{ecuacion en grado 2 para M para  b y c}.
\end{proof}

\begin{notation}
  We define $E_j$ to be the infinite standard basis (row) vector, e.g., $E_0=(1,0,\dots)$, $E_1=(0,1,0,\dots)$.
\end{notation}

\begin{lemma}\label{primera fila de M}
  Let the first row of $M\in L(K[x]^{\mathds{N}_0})$ be given by $M_{0*}=E_0$. If $M$ satisfies
  $$
    Y^k M=\sum_{i=0}^{k+1} a_i M^{k+1-i}Y^i,
  $$
  then $M_{kj}=a_j$ for $j=0,\dots, k+1$.
\end{lemma}

\begin{proof}
  Note that $(M^r)_{0*}=E_0$ for all $r$, and so $(M^rY^{i})_{0j}=M_{0,j-i}=\delta_{ij}$.
  Hence
  $$
    M_{kj} = (Y^k M)_{0j}=\sum_{i=0}^{k+1} a_i (M^{k+1-i}Y^i)_{0j} = \sum_{i=0}^{k+1} a_i \delta_{ij} = a_j,
  $$
  as desired.
\end{proof}

\begin{theorem} \label{clasificacion caso generico}
  Let $K$ be a field and let $b,c\in K$ with $c\ne 1$. Assume that $Q_n(b,c)\ne 0$ for all $n\in\mathds{N}$.
  Then $b$ and $c$ determine a unique twisting map via Proposition~\ref{bc determinan M}.
\end{theorem}

\begin{proof}
  We will use Corollary~\ref{graded condition}. For this we first prove that the matrix $M$ constructed in
  Proposition~\ref{M cumple ecuacion en grado dos}
  satisfies equality~\eqref{fundamental matrix equality} for all $k$.
  For $k=0$ this is clear, and
  from Lemma~\ref{primera fila de M} we obtain $M_{1*}=E_0+bE_1+cE_2$, hence, by~\eqref{ecuacion en grado 2 para M para b y c}, the
  equality~\eqref{condiciones twisting} holds for $k=1$.
  Assume by induction hypothesis
  that~\eqref{condiciones twisting} holds for $k<k_0$. Then Lemma~\ref{tecnico} and the fact that $M_{k_0,k_0+1}\ne 1$ yield
  $$
    Y^{k_0} M=\sum_{k=0}^{k_0+1}\frac{\overline{M}_{k_0,s}}{1-M_{k_0,k_0+1}} M^{k_0+1-s}Y^s.
  $$
  But then, by Lemma~\eqref{primera fila de M} we have $M_{k_0,s}=\frac{\overline{M}_{k_0,s}}{1-M_{k_0,k_0+1}}$ for $s=0,\dots,k_0+1$,
  which yields~\eqref{condiciones twisting} for $k=k_0$ and completes the inductive step. Finally,
  Corollary~\ref{graded condition} yields the desired twisting map, which is unique by Proposition~\ref{bc determinan M}.
\end{proof}

\begin{remark}\label{comparacion con ConnerGoetz}
  Combining Corollary~\ref{b y c determinan M} and Theorem~\ref{clasificacion caso generico} we obtain that
  $b$ and $c\ne 1$ determine a (necessarily unique) twisting map via Proposition~\ref{bc determinan M} if and only if
  $Q_n(b,c)\ne 0$ for all $n\in\mathds{N}$.

  This condition is the same as the condition used in~\cite{CG2}*{Theorem 3.4}. In order to verify this, we first note that
   the polynomials $f_n(a,b)$ used by~\cite{CG2} satisfy
  $f_n(a,b)=Q_{n+1}(b,a)$. In fact, since $Q_1(b,a)=f_0(a,b)=1$, $P_1(b,a)=e_0(a,b)=1$, and the recursive relations are the same, i.e.,
  $$
    \binom{P_n(b,a)}{Q_n(b,a)}=\begin{pmatrix}b&1\\ -a &1 \end{pmatrix} \binom{P_{n-1}(b,a)}{Q_{n-1}(b,a)}\quad\text{and}\quad
    \binom{e_n(a,b)}{f_n(a,b)}=\begin{pmatrix}b&1\\ -a &1 \end{pmatrix} \binom{e_{n-1}(a,b)}{f_{n-1}(a,b)},
  $$
  we conclude $e_n(a,b)=P_{n+1}(b,a)$ and $f_n(a,b)=Q_{n+1}(b,a)$, as desired.

  Note that by Remark~\ref{a distinto de cero}, when $a\ne 0$, the twisting map corresponding to
  $$
    M_{10}=a, \quad M_{11}=b,\quad\text{and}\quad M_{12}=1
  $$
  is equivalent to a twisting map with
  $$
    M_{10}=1, \quad M_{11}=b,\quad\text{and}\quad M_{12}=a,
  $$
  and so our results match the results of~\cite{CG2}.
\end{remark}

\section{Roots of $Q_n$}

In view of Theorem~\ref{clasificacion caso generico}, we want to analyze the polynomials $Q_n$ and their roots.
In particular we are interested in the following question: Given a pair $(b,c)\in K^2$, does there exists an $n\in\mathds{N}$ such that $Q_n(b,c)=0$?
If the answer is no, then $(b,c)$ defines a unique twisting map via the previous theorem. Else, if $(b,c)\ne (-1,1)$,
there is no twisting map for that $(b,c)$.

For a fixed pair $(b,c)$, from the recursive relations~\eqref{recursion} in matrix form, we obtain
$$
  \binom{P_n(b,c)}{Q_n(b,c)}=\begin{pmatrix}b&1\\ -c &1 \end{pmatrix}^n \binom{1}{1}.
$$
If the eigenvalues of $D:=\begin{pmatrix}b&1\\ -c &1 \end{pmatrix}$ are different, then
 there exists an invertible matrix $T$ such that
$$
  D= T \begin{pmatrix}\lambda_1&0\\ 0 &\lambda_2 \end{pmatrix} T^{-1},
$$
and so
$$
  Q_n=r_1 \lambda_1^n+r_2 \lambda_2^n
$$
for some $r_1,r_2\in K$. If $r_1,\lambda_2\ne 0$ then $Q_n=0$ if and only if
\begin{equation}\label{chequeo}
  \left(\frac{\lambda_1}{\lambda_2}\right)^n=-\frac {r_2}{r_1}.
\end{equation}

This condition is easier to verify than the infinite number of evaluations $Q_n(b,c)$.
 For example, if $K\subset \mathds{C}$, and $| \frac {r_1}{r_2}|\ne 1$, one can check the equality $Q_n=0$
 using real logarithms on the modulus in order to find the (unique) possible $n$, and then verifying the equality~\eqref{chequeo} for that $n$.

The eigenvalues of $D$ are
$$
  \lambda_1=\frac{1}{2} \left(b+1+\sqrt{(b-1)^2-4 c}\right)\quad\text{and}\quad   \lambda_2=\frac{1}{2} \left(b+1-\sqrt{(b-1)^2-4 c}\right),
$$
and from $Q_0=1=r_1+r_2$ and $Q_1=1=r_1\lambda_1 +r_2\lambda_2$ we obtain
$$
  r_1=\frac{1-\lambda_2}{\lambda_1-\lambda_2}\quad\text{and}\quad r_2=\frac{1-\lambda_1}{\lambda_2-\lambda_1}.
$$

If
\begin{equation}\label{casos}
  0\notin\{\lambda_1,\lambda_2,1-\lambda_1,1-\lambda_2,\lambda_1-\lambda_2\},
\end{equation}
then $Q_n=0$ if and only if
$$
  \left(\frac{\lambda_1}{\lambda_2}\right)^n=\frac{1-\lambda_1}{1-\lambda_2}.
$$

Note that if $b,c\in\mathds{R}$, then $4c<(b-1)^2$ implies $|\frac {r_1}{r_2}|\ne 1$, and so in this case it
can be determined if $Q_n=0$ for some $n$.

Now we give a detailed account of each exceptional case in~\eqref{casos}:

If the eigenvalues coincide ($\lambda_1-\lambda_2=0$) then $c=\frac{(b-1)^2}{4}$.
In that case $\lambda=\lambda_1=\lambda_2=\frac{b+1}{2}\ne 0$, since $\lambda=0$ leads to $c=1$ and we also have
$$
  Q_{n+1}=(n+1-n\lambda)\lambda^n.
$$
Hence $Q_{n+1}=0$ if and only if  $\lambda=\frac{n+1}{n}$ if and only if $b=1+\frac 2n$.

If one of the values $\lambda_1,\lambda_2$ is zero, then $\det(D)=b+c=0$. In that case~\eqref{formula para c} yields $c_n=c\ne 1$ for all $n$.

Note that $(1-\lambda_1)(1-\lambda_2)=c$,
and so, if $1-\lambda_1=0$ or $1-\lambda_2=0$, then $c=0$, and in that case $c_n=0\ne 1$ for all $n$.

This covers all cases of~\eqref{casos}. However there are some other interesting cases.

For example if we require $b=0$, then we recover the polynomials $S_n$ in~\cite{CG} via the equality $S_n(c)=Q_{n-1}(0,c)$.

Another exceptional case happens when $\lambda_1=-\lambda_2$. In that case $b=-1$, and then
$$
  c_n=\left\{ \begin{array}{ll} c&\text{if $n$ is odd}\\ 0 &\text{if $n$ is even.}\end{array}\right.
$$

\section{The case $yx=x^2-xy+y^2$}

In this section we assume that $\sigma$ is a twisting map and that $Y$ and $M$ are as in Corollary~\ref{graded condition}.
As before we write $M_{1,*}=(1,b,c,0,\dots)$ and assume that $(b,c)=(-1,1)$, which is the only case not covered
by Theorem~\ref{clasificacion caso generico}. This means that we are dealing with the commutation rule
$$
  yx=x^2-xy+y^2.
$$
By Corollary~\ref{graded condition} we have
$$
  YM=M^2-MY+Y^2,
$$
which implies $\mm^2=0$, where $\mm=M-Y$. The matrix $\widetilde{M}:= M-Y=\psi(x-y)\in L(K^{\mathds{N}_0})$
plays a central role in the classification of all the twisting maps with $(b,c)=(-1,1)$.
Note that $\mm_{0j}=\delta_{0j}-\delta_{1j}$
and that $\widetilde{M}_{1*}=\widetilde{M}_{0*}$.

\begin{remark}\label{condiciones de corolario pero para mm}
  Let $\mm\in L(K^{\mathds{N}_0})$ be such that $\mm_{0j}=\delta_{0j}-\delta_{1j}$ and $\mm_{kj}=0$ for $j>k+1$.
  Then a straightforward computation shows that $M:=\mm-Y$ determines a twisting map via Corollary~\ref{graded condition},
  if and only if for all $k\in\mathds{N}$ we have
  \begin{equation}\label{fundamental matrix equality for mm}
    Y^k \mm=\sum_{j= 0}^{k+1} \mm_{kj}M^{k+1-j} Y^j.
  \end{equation}
\end{remark}

\begin{lemma}\label{mixtos}
  Let $d\ne 0$. Then
  \begin{enumerate}
    \item $\mm_{k*}=\mm_{k-1,*}$ if and only if $\mm Y^{k-1}\mm = 0$.
    \item $\mm_{k*}=d \mm_{k-1,*}$ if and only if $\mm Y^{k-1}\mm = \frac{1-d}d Y^k \mm$.
    \item $\mm_{k*}=d \mm_{k-2,*}$ if and only if $\mm Y^{k-1}\mm +Y \mm Y^{k-2}\mm = \frac{1-d}d Y^k \mm$.
  \end{enumerate}
\end{lemma}

\begin{proof}
  We only prove~(2) and~(3), since~(1) follows from~(2) with $d=1$.
  Assume $\mm_{k*}=d \mm_{k-1,*}$. Then
  $$
    Y^{k}\mm=\sum_{j=0}^{k+1}(d \mm_{k-1,j})M^{k+1-j}Y^j=dM \sum_{j=0}^{k} \mm_{k-1,j}M^{k-j}Y^j=d M Y^{k-1}\mm=d(\mm+Y)Y^{k-1}\mm,
  $$
  where the first and the third equality follow from~\eqref{fundamental matrix equality for mm}. Now
$Y^{k}\mm=d(\mm+Y)Y^{k-1}\mm$ implies $(1-d)Y^k\mm=d \mm Y^{k-1}\mm$, and then $\mm Y^{k-1}\mm = \frac{1-d}d Y^k \mm$ follows.
On the other hand, $\mm Y^{k-1}\mm = \frac{1-d}d Y^k \mm$ implies $Y^{k}\mm=d(\mm+Y)Y^{k-1}\mm$ and then the first row of the matrix equality
  $$
    \sum_{j=0}^{k+1} \mm_{kj}M^{k+1-j}Y^j=\sum_{j=0}^k d\mm_{k-1,j}M^{k+1-j}Y^j
  $$
  is
  $$
    \sum_{j=0}^{k+1} \mm_{kj}E_j=\sum_{j=0}^k d\mm_{k-1,j}E_j,
  $$
  which yields $\mm_{k*}=d \mm_{k-1,*}$.

  Similarly $\mm_{k*}=d \mm_{k-2,*}$ if and only if $Y^{k}\mm=d M^2 Y^{k-2}\mm=d(\mm Y+Y\mm +Y^2)Y^{k-2}\mm$ if and only if
  $(1-d)Y^k\mm=d \mm Y^{k-1}\mm + d Y \mm Y^{k-2}\mm$, as desired.
\end{proof}

\begin{lemma}\label{mixtos 2}
  Let $n\in\mathds{N}$ with $n\ge 2$. Then $\widetilde{M}_{k*}=\widetilde{M}_{0*}$ for $0< k< n$ if and only if
  $\mm Y^{k}\mm=0$ for $0\le k\le n-2$. Moreover, in this case
  $$
    M^{k+1}=
    \begin{cases} Y^{k+1}+\sum_{j=0}^k Y^j \mm Y^{k-j} &\text{for $0\le k \le n-1$}\\
                          Y^{n+1} + \sum_{j=0}^n Y^j \mm Y^{n-j} +\mm Y^{n-1}\mm &\text{for $k=n$}.
    \end{cases}
  $$
\end{lemma}

\begin{proof}
  The first assertion follows directly from Lemma~\ref{mixtos}(1). In that case, expand $M^{k+1}=(\mm+Y)^{k+1}$. If $k<n$, no
  term in the expansion can have two times the factor $\mm$, and in the expansion of
  $M^{n+1}=(\mm+Y)^{n+1}$, the only term with two times the factor $\mm$ is $\mm Y^{n-1}\mm$.
\end{proof}

\begin{proposition}\label{caso particular}
  There exists a twisting map such that $\widetilde{M}_{k*}=\widetilde{M}_{0*}$ for all $k\in \mathds{N}$.
\end{proposition}
\begin{proof}
  Consider the matrix $\mm$ such that for all $j$,  $\mm_{j0}=1$, $\mm_{j1}=-1$ and $\mm_{ji}=0$ for $i>1$.
  We will show that $\mm$ satisfies the conditions of Remark~\ref{condiciones de corolario pero para mm}.
  Clearly $\mm_{0j}=\delta_{0j}-\delta_{1j}$ and $\mm_{kj}=0$ for $j>k+1$. So
  we have to prove that for all $k$
  \begin{equation}\label{igualdad en caso particular}
    Y^k \mm=M^{k+1}-M^k Y=M^k \mm.
  \end{equation}
  But clearly $Y^k\mm=\mm$ for all $k$, since the columns are constant. So we have to prove that $M^k\mm=\mm$. For $k=1$ we have
  $$
    M\mm=(\mm+Y)\mm=Y\mm=\mm,
  $$
  since $\mm^2=0$. This implies $M^k\mm=\mm$, which
  proves~\eqref{igualdad en caso particular},
  finishing the proof of the proposition.
\end{proof}

\begin{proposition}
  Let $n\in\mathds{N}$ with $n\ge 2$ and assume that $\widetilde{M}_{k*}=\widetilde{M}_{0*}$ for $1< k< n$ and
  that $\widetilde{M}_{n*}\ne \widetilde{M}_{0*}$. Then $\widetilde{M}_{nj}=0$ for $1<j<n$. Moreover, if one sets
  $m_i:=\widetilde{M}_{ni}$ then
  \begin{equation}\label{eq38}
    (1-m_0)Y^n \mm=m_0 \mm Y^{n-1}\mm+(m_0+m_1)\sum_{j=0}^{n-1}Y^j \mm Y^{n-j}+m_n \mm Y^n+(m_0+m_1+m_n+m_{n+1})Y^{n+1}
  \end{equation}
  and
  \begin{equation}\label{primera ecuacion}
    (m_{n+1}+1)(m_0+m_1)+m_n+m_{n+1}=0.
  \end{equation}
\end{proposition}

\begin{proof}
  Assume first that $m_{n+1}=\mm_{n,n+1}=0$. We first prove that $\mm_{nj}=0$ for $n\ge j>1$. Clearly $0=(\mm^2)_{nn}=(\mm_{nn})^2$, so
  $\mm_{nn}=0$. Now assume that $2<k\le n$ and that $\mm_{nj}=0$ for $j\ge k$. Then $0\le n-k+1\le n-2$ and so by Lemma~\ref{mixtos 2}
  we have $\mm Y^{n-k+1}\mm=0$. But then $0=(\mm Y^{n-k+1}\mm)_{k-1,k-1}=(\mm_{n,k-1})^2$, hence $\mm_{n,k-1}=0$. Thus inductively we obtain
  $\mm_{nj}=0$ for $n\ge j>1$.

  From $(\mm^2)_{n0}=0$ it follows that $m_0+m_1=0$, so~\eqref{eq38} reads $(1-m_0)Y^n \mm =m_0 \mm Y^{n-1}\mm$, which is satisfied by
  Lemma~\ref{mixtos}(2),
  and~\eqref{primera ecuacion} is also trivially satisfied in this case.

  Now we can assume that $m_{n+1}\ne 0$. Consider the equality~\eqref{fundamental matrix equality for mm} for $Y^n\mm$ and expand the power of $M$
  in each summand according to Lemma~\ref{mixtos 2}. We obtain
  \begin{align}
    Y^n\mm &=m_0 \mm Y^{n-1}\mm +\sum_{k=0}^{n+1} m_{k}Y^{n+1}+ \sum_{k=0}^{n} m_{k} \sum_{j=0}^{n-k} Y^j\mm Y^{n-j}\nonumber \\
     &=m_0 \mm Y^{n-1}\mm +\sum_{k=0}^{n+1} m_{k}Y^{n+1}+ \sum_{j=0}^{n} \left( \sum_{k=0}^{n-j}m_{k}\right) Y^j\mm Y^{n-j}
    \label{igualdad en nivel n}.
  \end{align}
  We will evaluate the matrix equality~\eqref{igualdad en nivel n} at the entries $(i,i+n+1)$ for $i=1,\dots,n-1$.

  First we claim that $(\mm Y^{n-1}\mm)_{i,i+n+1}=0$ for $i=1,\dots,n-1$. In fact,
  $$
    (\mm Y^{n-1}\mm)_{i,i+n+1}=\mm_{i*}\cdot (Y^{n-1}\mm)_{*,i+n+1}=(1,-1,0,\dots,0,\dots)\cdot (\mm_{n-1,i+n+1},\mm_{n,i+n+1},\dots).
  $$
  But, since $\mm_{jk}=0$ if $k>j+1$, we also have $\mm_{n-1,i+n+1}=0=\mm_{n,i+n+1}$, which proves the claim.

  Clearly $(Y^{n+1})_{i,i+n+1}=1$, so it remains to compute $( Y^j\mm Y^{n-j})_{i,i+n+1}=\mm_{i+j,i+j+1}$.
  Now we assert that
  \begin{equation}\label{primera diagonal}
    \mm_{i+j,i+j+1}=m_{n+1}\delta_{j,n-i}.
  \end{equation}

  In fact, since $0<i\le n-1$ and $0\le j\le n$, we have $i+j>0$, and so, for $1\le i+j\le n$, we know that~\eqref{primera diagonal}
  holds. So it suffices to prove that
  \begin{equation}\label{diagonal se anula}
    \mm_{n+k,n+k+1}=0\quad\text {for $1\le k\le n-1$.}
  \end{equation}
  But
  $$
    0=(\mm Y^{k-1}\mm)_{n,n+k+1}=\mm_{n,n+1}\mm_{n+k,n+k+1}=m_{n+1} \mm_{n+k,n+k+1},
  $$
  and so, since $m_{n+1}\ne 0$, we obtain~\eqref{diagonal se anula}, which proves~\eqref{primera diagonal}.

  Finally note that by~\eqref{diagonal se anula} we have
  $$
    (Y^{n}\mm)_{i,i+n+1}=\mm_{i+n,i+n+1}=0\quad\text{for $i=1,\dots,n-1$.}
  $$
  Gathering the entries at $(i,i+n+1)$ for all the terms of~\eqref{igualdad en nivel n}
  we obtain
  $$
    0=\sum_{k=0}^{n+1} m_{k}+\sum_{j=0}^{n} \left( \sum_{k=0}^{n-j}m_{k}\right)m_{n+1}\delta_{j,n-i}=
    \sum_{k=0}^{n+1} m_{k}+ m_{n+1} \sum_{k=0}^{i}m_{k}\quad\text{for $i=1,\dots,n-1$.}
  $$
  Subtracting these equalities for consecutive values of $i$ yields $m_{n+1}m_j=0$ for $j=2,\dots,n-1$, hence  $\widetilde{M}_{nj}=m_j=0$ for
  $1<j<n$.
  From the case $i=1$ we obtain $\sum_{k=0}^{n+1} m_{k}+ m_{n+1}(m_0+m_1)=0$, which gives~\eqref{primera ecuacion}.
  Finally, using $m_j=0$ for $1<j<n$ the equality~\eqref{eq38} follows directly from~\eqref{igualdad en nivel n}.
\end{proof}

\begin{proposition}\label{clasificacion}
  Let $n\in\mathds{N}$ with $n\ge 2$ and assume that $\widetilde{M}_{k*}=\widetilde{M}_{0*}$ for $1< k< n$ and
  that $\widetilde{M}_{n*}\ne \widetilde{M}_{0*}$. Rename the only possibly non zero entries in the $n$th row as
  $$
    a:=m_{n+1}=\widetilde{M}_{n,n+1},\quad b:=m_n=\widetilde{M}_{nn},\quad c:=m_1=\widetilde{M}_{n1}\quad\text{and}\quad d:=m_0=\widetilde{M}_{n0}.
  $$
  Then either
  \begin{enumerate}
    \item $(d,c,b,a)=(d,-d,-a,a)$ with $(d,a)\ne (1,0)$ or
    \item $(d,c,b,a)=(d,-1,0,a)$ with $a\ne 0$ and $d(a+1)=1$.
  \end{enumerate}
\end{proposition}

\begin{proof}
  We will prove the following four assertions:
  \begin{enumerate}
    \item[i)] If $a=0$, then $b=0$ and $c=-d$.
    \item[ii)] If $a\ne 0$ and $d=1$, then $b=-a$ and $c=-1$.
    \item[iii)] If $a\ne 0$, $d\ne 1$ and $b\ne 0$,  then $b=-a$ and $c=-d$.
    \item[iv)]  If $a\ne 0$, $d\ne 1$ and $b = 0$,  then $c=-1$ and $d(a+1)=1$.
  \end{enumerate}
  Note that in item~i) we have $(a,d)\ne (0,1)$, since $\mm_{n*}\ne \mm_{0*}$.
  On one hand~i), ii) and~iii) imply condition~(1) and on the other hand~iv) implies condition~(2). Since items~i)--iv) cover all possible cases,
  it suffices to prove these items in order to show that one of the conditions~(1) or~(2) necessarily holds.

  \medskip

  i): If $\mm_{n,n+1}=0$, then $0=(\mm^2)_{nn}=(\mm_{nn})^2$, so $b=\mm_{nn}=0$. Now we obtain
  $$
    0=(\mm^2)_{n0}=\mm_{n0}+\mm_{n1}=c+d.
  $$

  ii): If $a\ne 0$ and $d=1$, then the matrix equality~\eqref{eq38} at the entry $(1,2)$ yields
  $$
    0=(\mm Y^{n-1}\mm)_{12}+(1+c)((Y^{n-2}\mm Y^2)_{12}+(Y^{n-1}\mm Y)_{12}) + b(\mm Y^{n})_{12},
  $$
  since for any matrix $C$ we have $(C Y^3)_{12}=0$. So
  $$
    0=(1,-1,0,\dots)\cdot(\mm_{n-1,2},\mm_{n2},*,*,\dots)+(1+c)(\mm_{n-1,0}+\mm_{n1}) + b(\mm Y^n)_{12}.
  $$
  If $n=2$, this gives $0=-b+(1+c)(1+c)+b\mm_{10}=(1+c)^2$ and if $n>2$ this yields directly $0=(1+c)^2$.

  So $c=-1$ and then~\eqref{eq38} reads
  $$
    0=\mm Y^{n-1}\mm+b \mm Y^n+(a+b)Y^{n+1},
  $$
  and multiplying by $\mm$ from the left we obtain $0=(a+b)\mm Y^{n+1}$ since $\mm^2=0$. Hence $b=-a$, which concludes the proof of~ii).

  \medskip

  For the rest of the proof we assume $a\ne 0$ and $d=m_0\ne 1$ and we claim that
  \begin{equation}\label{valores}
    \mm_{n+1,0}=d\quad\text{and}\quad \mm_{n+1,1}=-d.
  \end{equation}
  For this we evaluate~\eqref{eq38} at the entry $(1,0)$, noting that
  $(Y^n\mm)_{10}=\mm_{n+1,0}$, $(\mm Y^{n-1}\mm)_{10}=1-m_0$ and that $(Y^j\mm Y^{n-j})_{10}=0$ if $j<n$. So we obtain
  $$
    (1-m_0)\mm_{n+1,0}=m_0(1-m_0),
  $$
  and since $m_0\ne 1$, we have $\mm_{n+1,0}=m_0=d$.

  Now we compute $\mm_{n+1,1}$. For this we evaluate~\eqref{eq38} at the entry $(1,1)$, noting that
  $(Y^n\mm)_{11}=\mm_{n+1,1}$, $(\mm Y^{n-1}\mm)_{11}=-1-m_1$, $(Y^{n-1}\mm Y)_{11}=\mm_{n0}=m_0$ and that $(Y^j\mm Y^{n-j})_{11}=0$ if $j<n-1$. So
  $$
    (1-m_0)\mm_{n+1,1}=m_0(-1-m_1)+(m_0+m_1)m_0=-m_0(m_0+m_1),
  $$
  and since $m_0\ne 1$, we have $\mm_{n+1,1}=-m_0$, concluding the proof of~\eqref{valores}.

  \medskip

  Now, the equalities $0=(\mm^2)_{n0}=(\mm^2)_{n1}$ give
  \begin{equation} \label{eq1}
    0=d+c+db+da=-d-c+cb-da.
  \end{equation}
  Adding these yields $(c+d)b=0$, and so, if $b\ne 0$, then $c=-d$. Moreover,~\eqref{primera ecuacion} reads
    \begin{equation} \label{eq3}
    (a+1)(c+d)+b+a=0,
    \end{equation}
     and so $b=-a$, proving~iii).

  \medskip

  Finally, if $b=0$, then~\eqref{eq1} yields
  \begin{equation}\label{segunda ecuacion}
    ad+c+d=0
  \end{equation}
  and from~\eqref{eq3} it follows that $ac+a=0$, hence $c=-1$. Then~\eqref{segunda ecuacion} implies $d(a+1)=1$, which concludes the proof of~iv).
\end{proof}

\section{The family $A(n,d,a)$}
\label{seccion 5}
In this section we will describe the case~(1) of Proposition~\ref{clasificacion}. We will prove that the resulting twisting map depends
only on $n$, $d$ and $a$. We obtain a family of twisted tensor products $A(n,d,a)$, parameterized by $n\in\mathds{N}$, $n\ge 2$, and $(a,d)\in K^2$,
such that for an infinite family of polynomials $R_k$ (see Definition~\ref{definicion Rk}) we have $R_k(a,d)\ne 0$.

\begin{remark}
   Let $\tau$ be a twisting map, assume that $Y$ and $M$ are as in Corollary~\ref{graded condition} and set $\mm=M-Y$.
  Let $n\in\mathds{N}$ with $n\ge 2$, take $a,d\in K$ with $(d,a)\ne (1,0)$, assume that $\mm_{j,*}=\mm_{0,*}$ for $j<n$, and that
 we are in the case~(1) of Proposition~\ref{clasificacion}, i.e.,
 $$
 Y^n=dM^{n+1}-dM^nY-a MY^n+(a+1)Y^{n+1},
 $$
 which we write as
 $$
 Y^n\mm=dM^n\mm-a\mm Y^n.
 $$
 Using Lemma~\ref{mixtos 2}(1) we obtain $M^n\mm=\mm Y^{n-1}\mm+Y^n\mm$, and so
 \begin{equation}\label{nivel uno}
    d\mm Y^{n-1}\mm=e Y^n \mm+a \mm Y^n,
  \end{equation}
  where $e:=1-d$.
\end{remark}

\begin{proposition}\label{inductivo Anda}
  Let $A$ be an associative $K$-algebra, $a,d\in K$ with $(d,a)\ne (1,0)$, $\mm,Y\in A$ satisfying~\eqref{nivel uno}.
   Then for all $k\ge 1$ we have
  \begin{equation}\label{nivel k}
    d_k\mm Y^{kn-1}\mm=e^k Y^{kn} \mm-(-a)^k \mm Y^{kn},
  \end{equation}
  where
  \begin{equation}\label{d sub k}
    d_k=d\sum_{j=0}^{k-1} e^j (-a)^{k-1-j}=de^{k-1}[k]_{-a/e}
  \end{equation}
\end{proposition}

\begin{proof}
  If $d=0$, then $e=1$ and~\eqref{nivel uno} reads $Y^n\mm=-a \mm Y^n$. A direct computation shows that then $Y^{kn}\mm=(-a)^k \mm Y^{kn}$, which
  is~\eqref{nivel k} in this case.

  Now assume $d\ne 0$. Then a straightforward computation shows that
  \begin{equation}\label{ecuacion dk}
    e^k-\frac{a d_k}{d}=\frac{d_{k+1}}{d}.
  \end{equation}
  Now we proceed by induction on $k$. For $k=1$, equality~\eqref{nivel k} is just~\eqref{nivel uno}.
  Assume that~\eqref{nivel k} holds for some $k$. Multiplying~\eqref{nivel k} by $\mm Y^{n-1}$ from the left yields
  $$
    d_k \mm Y^{n-1}\mm Y^{kn-1}\mm=e^k \mm Y^{n(k+1)-1}\mm-(-a)^k\mm Y^{n-1}\mm Y^{kn}.
  $$
  Replacing $\mm Y^{n-1}\mm$ using~\eqref{nivel uno} and changing sides we obtain
  {\small $$
    e^k  \mm Y^{n(k+1)-1}\mm-(-a)^k\left(\frac{e}{d}Y^n \mm Y^{kn}+\frac ad \mm Y^{(k+1)n}\right)=
        d_k\left(\frac{e}{d}Y^n \mm Y^{kn-1}\mm+\frac ad \mm Y^{(k+1)n-1}\mm \right)
  $$}
  and by the inductive hypothesis we get
  {\small $$
    \left( e^k-\frac{a d_k}{d}\right)\mm Y^{n(k+1)-1}\mm = \frac {eY^n}d\left( e^k Y^{kn} \mm-(-a)^k \mm Y^{kn}  \right)+
    (-a)^k\left(\frac{e}{d}Y^n \mm Y^{kn}+\frac ad \mm Y^{(k+1)n}\right).
  $$}
  From this and~\eqref{ecuacion dk} it follows that
  $$
    \frac{d_{k+1}}d \mm Y^{n(k+1)-1}\mm=  \frac {e e^k}d  Y^{(k+1)n} \mm+\frac {a(-a)^k}d \mm Y^{(k+1)n},
  $$
  and clearing denominators completes the induction step and concludes the proof.
\end{proof}

\begin{theorem}\label{anda}
  Let $\tau$ be a twisting map, assume that $Y$ and $M$ are as in Corollary~\ref{graded condition} and set $\mm=M-Y$.
  Let $n\in\mathds{N}$ with $n\ge 2$, take $a,d\in K$ with $(d,a)\ne (1,0)$, assume that $\mm_{j,*}=\mm_{0,*}$ for $j<n$, and that
  $a$, $d$, $Y$ and $\mm$ satisfy~\eqref{nivel uno}.
  Let $d_k$ be defined by~\eqref{d sub k}. Then  $d_k+e^k\ne 0$ for all $k\in \mathds{N}$,
    \begin{equation}\label{el siguiente es igual}
     \mm_{kn,*}=\mm_{kn+j,*}\quad\text{for $k\ge 1$ and $0< j<n$,}
  \end{equation}
  and
  \begin{equation}\label{coeficientes}
    Y^{kn}\mm=\sum_{i=0}^k c_{k,i}M^{(k-i)n}\mm Y^{in},
  \end{equation}
  where $c_{k,0}=d\prod_{i=2}^{k}\frac{d_i}{d_i+e^i}$ and $c_{k,r}=\frac{(-a)^r}{d_r+e^r}\prod_{i=r+1}^{k}\frac{d_i}{d_i+e^i}$ for $1\le r\le k$.
\end{theorem}

\begin{proof}

  We first prove~\eqref{el siguiente es igual}.
  Let $i\in\{0,\dots,n-2\}$. If $a\ne 0$, then multiplying the equality~\eqref{nivel k} by $\mm Y^i$ from the left
  yields $\mm Y^{kn+i}\mm=0$ and similarly, if $e\ne0$, then multiplying~\eqref{nivel k} by $Y^i \mm $ from the right
  yields $\mm Y^{kn+i}\mm=0$. Since $(e,a)\ne (0,0)$ we obtain $\mm Y^{kn+i}\mm=0$ for $i=0,\dots,n-2$ and Lemma~\ref{mixtos}(1)
  implies~\eqref{el siguiente es igual}.

  Now we assume $d,a\ne 0$ and prove $d_k+e^k\ne 0$ and~\eqref{coeficientes} by induction on $k$.
  For $k=1$ we use the equality  $Y^{n-1}\mm=M^{n-1}\mm$ and~\eqref{nivel uno} to obtain
  \begin{align*}
    M^n\mm&=M Y^{n-1}\mm=\mm Y^{n-1}\mm+Y^n \mm\\
  &=\frac ed Y^n \mm+\frac ad \mm Y^n+Y^n \mm\\
  &= \frac{e+d}{d}Y^n \mm+\frac ad \mm Y^n.
  \end{align*}
  Since $e+d=1$, then clearing denominators and solving for $Y^n\mm$ gives~\eqref{coeficientes} for $k=1$
  (note that $c_{1,0}=d$ and $c_{1,1}=-a$, as the empty product takes the value 1).

  Now assume that~\eqref{coeficientes} holds for $k-1\ge 1$ and that $e^i+d_i\ne 0$ for $i<k$.
  By~\eqref{nivel k} we have
  $$
    e^k Y^{kn} \mm= d_k\mm Y^{kn-1}\mm + (-a)^k \mm Y^{kn}=-d_k Y^{kn}\mm+d_k M Y^{kn-1}\mm +(-a)^k \mm Y^{kn},
  $$
  and so
  $$
    (e^k+d_k)Y^{kn}\mm=d_k M Y^{kn-1}\mm +(-a)^k \mm Y^{kn}.
  $$
  By the inductive hypothesis and~\eqref{el siguiente es igual} we have
  $$
    Y^{kn-1}\mm=M^{n-1}\sum_{j=0}^{k-1}c_{k-1,j}M^{n(k-1-j)}\mm Y^{nj}
  $$
  and so
  \begin{equation}\label{coeficientes induccion}
    (e^k+d_k)Y^{kn}\mm=d_k \sum_{j=0}^{k-1}c_{k-1,j}M^{n(k-j)}\mm Y^{nj}+(-a)^k \mm Y^{kn}.
  \end{equation}
  Evaluating the matrix equality~\eqref{coeficientes induccion} at the entry $(0,kn)$ yields $e^k+d_k\ne 0$. In fact, since
  $(M^{n(k-j)})_{0*}=(1,0,0,\dots,0,\dots)$ for $j<n$, we have $(M^{n(k-j)}\mm Y^{nj})_{0,kn}=0$ for $j<n$ (note that $\mm Y^{nj}$ is of degree
  $-nj$), which implies
  $$
    (e^k+d_k)(Y^{kn}\mm)_{0,kn}=(-a)^k(\mm Y^{kn})_{0,kn}=(-a)^k \mm_{00}=(-a)^k\ne 0.
  $$
  The equality~\eqref{coeficientes induccion} also implies that~\eqref{coeficientes} holds for $k$ with
  $$
    c_{k,k}=\frac{(-a)^k}{e^k+d_k}\quad\text{and}\quad c_{k,j}=\frac{d_k c_{k-1,j}}{e^k+d_k}\quad\text{for $j<k$.}
  $$
  This completes the induction step and concludes the proof in the case $a,d\ne 0$.

  If $d=0$, then $d_k=0$ and  $e^k=1$ for all $k$, $c_{k,j}=0$ for $j<k$, and from $Y^n \mm=-a \mm Y^n$ it follows that
  $$
    Y^{kn}\mm=(-a)^k \mm Y^{kn}=c_{k,k}  \mm Y^{kn},
  $$
  as desired.

  Finally, if $a=0$, then $\frac{e_k}{d_k}=\frac{1-d}d$, hence by~\eqref{nivel k}, Lemma~\ref{mixtos}(2) and~\eqref{el siguiente es igual}
  we obtain
  $$
    \mm_{kn,*}=d \mm_{kn-1,*}=d \mm_{(k-1)n,*},
  $$
  and so~\eqref{coeficientes} holds with $c_{k,0}=d^k$ and $c_{k,j}=0$ for $j>0$, which concludes the proof, since then $e^k+d_k=e^{k-1}\ne 0$
  (note that $e=0$ leads to the contradiction $(d,a)=(1,0)$).
\end{proof}

\begin{remark} \label{formulas explicitas del teorema}
Note that Theorem~\ref{anda} yields explicit formulas for the entries of $\mm$:
  $$
    \mm_{kn+j,0}=-\mm_{kn+j,1}=c_{k,0}=d\prod_{i=2}^{k}\frac{d_i}{d_i+e^i}\quad\text{for $k\ge 1$ and $0\le j<n$,}
  $$
  $$
    \mm_{kn+j,rn}=-\mm_{kn+j,rn+1}=c_{k,r}=\frac{(-a)^r}{e^r+d_r}\prod_{i=r+1}^{k}\frac{d_i}{d_i+e^i}\quad\text{for $k\ge r\ge 1$ and $0\le j<n$,}
  $$
  and all other entries of $\mm$ are zero.
\end{remark}

\begin{definition}\label{definicion Rk}
  For $a,d\in K$ and $k\in\mathds{N}$ we define the polynomial
  $$
    R_k(a,d):=(1-d)^k+d\sum_{j=0}^{k-1}(1-d)^j(-a)^{k-1-j}.
  $$
\end{definition}
Note that $R_k(a,d)=e^k+d_k$, where $d_k$ is defined in~\eqref{d sub k}. We have $R_1(a,d)=1$ and $R_2(a,d)=1-d-ad$, so $R_2(a,d)\ne 0$
implies $(a+1)d\ne 1$, in particular $(a,d)=(0,1)$ is not allowed if we require $R_2(a,d)\ne 0$.

\begin{corollary}\label{Corolario anda}
  Let $a,d\in K$ be such that for all $k$ we have $R_k(a,d)\ne 0$. Then the formulas in
  Remark~\ref{formulas explicitas del teorema} define a matrix  $\mm$ that determines a twisting map via
  Remark~\ref{condiciones de corolario pero para mm}.
\end{corollary}

\begin{proof}
  Note that
  $$
    M_{ij}=\begin{cases}
         \mm_{ij}, & \mbox{if }i\ge j  \\
         \mm_{ij}+1, & \mbox{if } i+1=j.
       \end{cases}
  $$
  Then $M_{0*}=(1,0,0,\dots )$, since $\mm_{0*}=(1,-1,0,0,\dots)$, and so by Corollary~\ref{graded condition} we have to prove that $M$
  satisfies~\eqref{fundamental matrix equality}
  for all $k$.

  For $k=0$ this is clear. For $k=1,\dots,n-1$ the equality~\eqref{fundamental matrix equality}
  reads $Y^k M=M^{k+1}-M^k Y+Y^{k+1}$, which is equivalent to
  $Y^k \mm=M^k \mm$. A straightforward computation as in Lemma~\ref{mixtos} shows that these equalities are satisfied if and only if
  $$
    \mm Y^{k-1}\mm =0\quad\text{for $k=1,\dots , n-1$.}
  $$
  We claim that
  \begin{equation}\label{igualdad en casi todos los niveles}
    \mm Y^{nk+j}\mm =0\quad\text{for $k\ge 0$ and $j=0,\dots , n-2$.}
  \end{equation}
  A similar computation as above shows that then it suffices to prove~\eqref{fundamental matrix equality} for all $k=rn$,
  because~\eqref{igualdad en casi todos los niveles}
  implies~\eqref{fundamental matrix equality} for all other $k$.
  Now we prove~\eqref{igualdad en casi todos los niveles}:

  Since the only non zero entries in $\mm_{k*}$ are of the form $\mm_{k,rn}$ or $\mm_{k,rn+1}$ for some $r\ge 0$, and
  $\mm_{k,rn+1}=-\mm_{k,rn}$, it suffices to verify that
  $$
    (Y^{nk+j}\mm)_{rn,l}=(Y^{nk+j}\mm)_{rn+1,l}\quad\text{for $j=0,\dots,n-2$.}
  $$
  But this is equivalent to $\mm_{(r+k)n+j,l}=\mm_{(r+k)n+j+1,l}$ for $j=0,\dots,n-2$, which holds by the definition
  of $\mm$ and so ~\eqref{igualdad en casi todos los niveles} is true.

  It only remains to prove~\eqref{fundamental matrix equality} for $k=rn$ with $r\ge 1$. A straightforward computation
  using~\eqref{igualdad en casi todos los niveles}
  shows that it suffices to prove~\eqref{nivel k} for all $k$, and by Proposition~\ref{inductivo Anda} we only have to prove that $\mm$
  satisfies~\eqref{nivel uno}.

  We will prove the equality~\eqref{nivel uno} in each entry $(l,nk+j)$.
  Since the columns $\mm_{*,nk+j}$ vanish for $j=2,\dots,n-1$ and $\mm_{*,nk}=-\mm_{*,nk+1}$, it suffices to prove~\eqref{nivel uno} at the entries
  $(l,nk)$.
  So we have to prove
  $$
    d(\mm Y^{n-1}\mm)_{l,nk}=e (Y^n \mm)_{l,nk}+a(\mm Y^n)_{l,nk}\quad\text{for $l\ge 0$ and $k\ge 0$.}
  $$
  But
  $$
    (\mm Y^{n-1}\mm)_{l,nk}=M_{l*}\cdot(Y^{n-1}\mm)_{*,nk},\quad (Y^n \mm)_{l,nk}=\mm_{l+n,nk},\quad (\mm Y^n)_{l,nk}=(\mm)_{l,n(k-1)}
  $$
  and $\mm_{l*}=\mm_{rn,*}$
  for $l=rn+j$  with $j=0,\dots, n-1$, hence it suffices to prove
  \begin{equation}\label{igualdad 518 en entradas}
    d(M_{rn,*}\cdot(Y^{n-1}\mm)_{*,nk})=e \mm_{(r+1)n,nk}+a \mm_{rn,n(k-1)}\quad\text{for all $r,k$.}
  \end{equation}
  Note that $\mm_{ij}=0$ if $i<0$ or $j<0$.

  By definition
  \begin{equation}\label{fila}
    \mm_{rn,*}=\sum_{i=0}^r c_{r,i}(E_{in}-E_{in+1}),
  \end{equation}
  where $E_j$ is the infinite vector with $(E_j)_i=\delta_{ij}$,
  \begin{equation}\label{coeficientes c}
    c_{k,0}=\prod_{i=1}^{k}\frac{d_i}{d_i+e^i} \quad\text{for $k\ge 0$,}\quad
    c_{k,r}=\frac{(-a)^r}{e^r+d_r}\prod_{i=r+1}^{k}\frac{d_i}{d_i+e^i} \quad \text{for $1\le r\le k$}
  \end{equation}
  and $c_{k,r}=0$ for all other $(k,r)$.

  Since $(Y^{n-1}\mm)_{*,nk}=(\mm_{n-1,nk},\mm_{n,nk},\mm_{n+1,nk},\dots)$ we have
  $$
    M_{rn,*}\cdot(Y^{n-1}\mm)_{*,nk}=\sum_{i=0}^r c_{r,i}(\mm_{in+n-1,nk}-\mm_{in+1+n-1,nk})=\sum_{i=0}^r c_{r,i}(c_{i,k}-c_{i+1,k}).
  $$
  Moreover,
  $$
    e \mm_{(r+1)n,nk}=e c_{r+1,k}\quad\text{and}\quad a \mm_{rn,n(k-1)}= a c_{r,k-1}\quad\text{(note that $c_{r,-1}=0$)}
  $$
  so~\eqref{igualdad 518 en entradas} reads
  \begin{equation}\label{igualdad 518 mayor que cero}
    d \sum_{i=0}^r c_{r,i}(c_{i,k}-c_{i+1,k})= e c_{r+1,k}+a c_{r,k-1}.
  \end{equation}
  In order finish the proof it suffices to prove~\eqref{igualdad 518 mayor que cero} for all $r,k$. For this
  we will use
  \begin{equation}\label{ecuacion util}
    e d_r+d (-a)^r=d_{r+1},
  \end{equation}
  which follows directly from the definitions of $d_r$.

  For $k>r+1$ the equality~\eqref{igualdad 518 mayor que cero} is trivially true, since in that case both sides vanish.
  If $k=r+1$, then~\eqref{igualdad 518 mayor que cero} reads

  \begin{equation}\label{caso r+1}
    -d c_{r,r}c_{r+1,r+1}=e c_{r+1,r+1}+a c_{r,r}.
  \end{equation}

  Since $c_{r,r}(e^r+d_r)=(-a)^r$ (note that $d_0=0$), this is equivalent to
  $$
    -d (-a)^{r}(-a)^{r+1}=e (-a)^{r+1}(e^r+d_r)+a (-a)^{r}(e^{r+1}+d_{r+1}).
  $$
  For $a=0$ this is true, and if $a\ne 0$, this is equivalent to $-d (-a)^r=e(e^r+d_r)-(e^{r+1}+d_{r+1})$, which
   follows directly from~\eqref{ecuacion util}, hence the case $k=r+1$ is proved.

  Now we can assume that $k\le r$ and we will use
  that for $i\le r$ we have
  \begin{equation}\label{cr inductivo}
    c_{r+1,i}=c_{r,i}\frac{d_{r+1}}{d_{r+1}+e^{r+1}}.
  \end{equation}

  We prove~\eqref{igualdad 518 mayor que cero} by induction on $r$ (assuming $k\le r$ and using that~\eqref{igualdad 518 mayor que cero} is true for
  $k=r+1$). For $r=0=k$ this means
  $$
    d c_{0,0}(c_{0,0}-c_{1,0})=  e c_{1,0}+a c_{0,-1}.
  $$
  Using $c_{0,0}=1$ and $c_{1,0}=d=1-e$ we see that this equality is equivalent to $d(1-d)=ed$ which is true by definition of $e$.

  Assume~\eqref{igualdad 518 mayor que cero} is true for some $r-1\ge 0$. Multiplying~\eqref{ecuacion util} by
  $e^{r+1}$ we obtain
  $$
    e d_r e^{r+1}+d (-a)^r e^{r+1}=e e^r d_{r+1},
  $$
  and adding $e d_r d_{r+1}$ this reads
  $$
    d_r e(e^{r+1}+d_{r+1})+d (-a)^r e^{r+1}=e d_{r+1}(e^r+d_r).
  $$
  Using that $(-a)^r=c_{r,r}(e_r+d_r)$  we obtain
  $$
    \frac{d_r}{e^r+d_r} e(e^{r+1}+d_{r+1})+d c_{r,r} e^{r+1}=e d_{r+1},
  $$
  which we can write as
  $$
    \frac{d_r}{e^r+d_r} e + d c_{r,r} \left(1-\frac{d_{r+1}}{d_{r+1}+e^{r+1}}\right)=e \frac{d_{r+1}}{d_{r+1}+e^{r+1}}.
  $$
  Next we multiply by $c_{r,k}$ and, since by~\eqref{cr inductivo} we know that $c_{r+1,k}=c_{r,k} \frac{d_{r+1}}{d_{r+1}+e^{r+1}}$,
  it follows that
  \begin{equation}\label{ecuacion intermedia}
    \frac{d_r}{e^r+d_r} e c_{r,k} + d c_{r,r} (c_{r,k}-c_{r+1,k})=e c_{r+1,k}.
  \end{equation}
  We claim that
  \begin{equation}\label{igualdad en la suma}
    e c_{r,k} = d \sum_{i=0}^{r-1}c_{r-1,i}(c_{i,k}-c_{i+1,k})-a c_{r-1,k-1}.
  \end{equation}
  In fact, if $k<r$, this follows from the inductive hypothesis, and if $k=r$,
  then~\eqref{caso r+1} gives the same equality. Now the equalities~\eqref{igualdad en la suma} and~\eqref{cr inductivo} imply
  $$
    \frac{d_r}{e^r+d_r} e c_{r,k}=
    \frac{d_r}{e^r+d_r} d \sum_{i=0}^{r-1}c_{r-1,i}(c_{i,k}-c_{i+1,k})-a c_{r,k-1}
    = d \sum_{i=0}^{r-1}c_{r,i}(c_{i,k}-c_{i+1,k})-ac_{r,k-1}.
  $$
  So the equality~\eqref{ecuacion intermedia}
  yields
  $$
    d \sum_{i=0}^{r-1}c_{r,i}(c_{i,k}-c_{i+1,k})-a c_{r-1,k-1}+ d c_{r,r} (c_{r,k}-c_{r+1,k})=e c_{r+1,k},
  $$
  which is~\eqref{igualdad 518 mayor que cero} for $r$. This completes the inductive step, proves~\eqref{igualdad 518 mayor que cero} and
  concludes the proof.
\end{proof}

\section{Roots of $R_k$}
\label{seccion 6}
In view of Corollary~\ref{Corolario anda}, we want to analyze the polynomials $R_k$ and their roots.
In particular we are interested in the following question: Given a pair $(a,d)$, does there exists  $k\in\mathds{N}$ such that $R_k(a,d)=0$?
If the answer is no, then for each $n\ge 2$ the pair $(a,d)$ defines a unique twisting map via Theorem~\ref{anda}. Else
there is no twisting map satisfying item~(1) of Proposition~\ref{clasificacion} for that $(a,d)$.

Fix $(a,d)$. If $-a=1-d=e$, then $d_k=kd e^{k-1}$. So $R_k(a,d)=e^k+kd e^{k-1}=e^{k-1}(kd+e)=0$ if and only if $0=kd+e=d(k-1)+1$. In that case
$a=d-1=-\frac{k}{k-1}$.

Now assume  $-a\ne 1-d=e$. Then
$$
  R_k(a,d)=e^k+d\frac{e^k-(-a)^k}{e+a}=\frac{e^{k+1}+a e^k+d e^k-d(-a)^k}{e+a},
$$
so $R_k(a,d)=0$ if and only if
$$
  0=e^{k+1}+a e^k+d e^k-d(-a)^k=(a+1)e^k-(1-e)(-a)^k.
$$
In that case $e\ne 0$ and $e\ne 1$, since $e=0$ leads to $a=0$ and $e=1$ leads to
$a=-1$, which contradicts $-a\ne e$. So  $R_k(a,d)=0$ if and only if
\begin{equation}\label{ecuacion fraccion}
  \frac{1+a}{1-e}=\frac{(-a)^k}{e^k}.
\end{equation}

This condition is much easier to handle than the original condition. Assume $K\subset\mathds{C}$.
If~\eqref{ecuacion fraccion} is satisfied and $\left|\frac{a}{e}\right|\ne 1$, then
$$
  k \log\left|\frac{a}{e}\right|=\log\left|\frac{1+a}{1-e}\right|.
$$
Moreover, if~\eqref{ecuacion fraccion} is satisfied and $\left|\frac{a}{e}\right|= 1$, then necessarily $\left|\frac{1+a}{1-e}\right|=1$,
and an elementary computation shows that then either $-a=e$ or $-a=\bar e$, where $\bar e$ is the complex conjugate of $e$.
The first case is impossible by assumption, and another elementary computation shows that~\eqref{ecuacion fraccion} is satisfied
if and only if $r=\frac{u^{2k}-1}{u^{2k-1}-u}$, where $u=\frac{e}{|e|}$ is a unitary complex number and $r=|e|$.
Hence, if $K\subset\mathds{C}$, we can describe a complete
strategy in order to determine if for a given pair $(a,d)$ we have
$R_k(a,d)\ne 0$ for all $k$.
\begin{enumerate}
  \item If $-a=1-d=:e$ then $R_k(a,d)\ne 0$ for all $k$ if and only if $(a,d)\ne \left(-\frac{k+1}{k},-\frac{1}{k}\right)$ for all $k\in\mathds{N}$.
  \item If $-a\ne 1-d=e$, then
  \begin{enumerate}
    \item[a)] if $\left|\frac{a}{e}\right|\ne 1$ then
    \begin{enumerate}
      \item[i)] if $\frac{\log\left|\frac{1+a}{1-e}\right|}{\log\left|\frac{a}{e}\right|}\notin \mathds{N}_{\ge 2}$, then $R_k(a,d)\ne 0$ for
          all $k$.
      \item[ii)] if $\frac{\log\left|\frac{1+a}{1-e}\right|}{\log\left|\frac{a}{e}\right|}=k_0\in \mathds{N}_{\ge 2}$, then $R_k(a,d)\ne 0$ for
          all $k$ if and only if $R_{k_0}(a,d)\ne 0$.
    \end{enumerate}
    \item[b)]  if $\left|\frac{a}{e}\right| = 1$ and $a\ne \bar e$, then $R_k(a,d)\ne 0$ for all $k$.
    \item[c)]  if $\left|\frac{a}{e}\right| = 1$ and $a= \bar e$, then $R_k(a,d)\ne 0$ for all $k$ if and only if $r\ne
        \frac{u^{2k}-1}{u^{2k-1}-u}$  for all
    $k\in \mathds{N}_{\ge 2}$, where $u=\frac{e}{|e|}$ and $r=|e|$.
  \end{enumerate}
\end{enumerate}

\begin{example}
  This algorithm determines that $d=0$, $a=-1$ is a valid choice, since then $e=1$ and we are in the case (1), with
  $(a,d)=(-1,0)\ne \left(-\frac{k+1}{k},-\frac{1}{k}\right)$ for all $k\in\mathds{N}$. One verifies that this example corresponds to~\cite{CG}*{Example 5.4}.
\end{example}

\section{The case $y^nx=d x^{n+1}-x^n y+(a+1)y^{n+1}$}
\label{seccion 7}
Sections~\ref{seccion 7} and~\ref{seccion familia} are dedicated to the analysis of the case~(2) of
Proposition~\ref{clasificacion}.
So in this three sections
$\sigma$ is a twisting map and $Y$ and $M$ are as in Corollary~\ref{graded condition} and $\mm=M-Y$.
Moreover there exists $n\in\mathds{N}$ with $n\ge 2$, such that $\mm_{k*}=\mm_{0*}$ for $1< k< n$ and
$\mm_{n*}=d E_0-E_1+a E_{n+1}$ for some $a,d\in K^{\times}$ with $d(a+1)=1$. This implies that we are dealing with the case
$$
  y^nx=d x^{n+1}-x^n y+(a+1)y^{n+1},\quad\text{with $a\ne 0,-1$}.
$$
The corresponding matrix equality is
\begin{equation}
\label{igualdad matricial para Y^n M tilde}
  Y^n\mm=d M^{n+1}-M^n Y+a Y^{n+1}.
\end{equation}
We will use this equality and similar matrix equalities in order to compute the different possibilities for the resulting twisting maps. There is only one choice for the first $2n-1$ rows, but four choices for
the $2n$'th row. In each of the four cases
the rows are determined until the row $3n-1$. One can determine the rows $3n$, $3n+1$ and $3n+2$, in each of the four
cases, and
obtain again four cases in each of them (so we have 16 cases).
As the number of possibilities grows, the systems of equations get more and more involved, so a full classification seems very difficult to achieve.

However, in section~\ref{seccion familia} we manage to describe a family of twisting maps, such that four of the above mentioned 16 cases coincide in
the first $3n+2$ rows with members of this family.

In the present section we will establish some technical formulas. On one hand with these technical results
 one can carry out the computations mentioned above in order to determine the possibilities for the first rows of $\mm$. On the other hand, in
 section~\ref{seccion familia}, they will allow us to describe a certain
family of twisting maps called $B(a,L)$. The following lemma is a result on lower infinite Hessenberg matrices, that should be well-known, but we couldn't find any reference to it in the literature.

\begin{lemma}\label{producto en la diagonal}
  Let $A\in L(K^{\mathds{N}_0})$ be an infinite matrix such that $A_{ij}=0$ for $j>i+1$. Then
  $$
  (A^{n+1})_{j,j+n+1}=\prod_{k=0}^n A_{j+k,j+k+1}.
  $$
\end{lemma}
\begin{proof}
  Since $w(A)=-1$, by Proposition~\ref{proposition valuacion} we have $w(A^{n+1})=-n-1$. A direct computation using
  Proposition~\ref{proposition valuacion},
  shows that then $(A^{n+1})^{(-n-1)}=(A^{(-1)})^{n+1}$, and so
  $$
    (A^{n+1})_{j,j+n+1}=((A^{n+1})^{(-n-1)})_{j,j+n+1}=\sum_{i_1,i_2,\dots,i_n}A^{(-1)}_{j,i_1}A^{(-1)}_{i_1,i_2}\cdots A^{(-1)}_{i_n,j+n+1}.
  $$
  But
  $$
  A^{(-1)}_{ij}=\begin{cases}
                     A_{i,i+1}, & \mbox{if } j=i+1 \\
                     0, & \mbox{otherwise},
                   \end{cases}
  $$
  hence the only term that survives in the sum is the term with $i_{k+1}=i_k+1$ for all $k$ and so
  $$
    (A^{n+1})_{j,j+n+1}=A^{(-1)}_{j,j+1}A^{(-1)}_{j+1,j+2}\cdots A^{(-1)}_{j+n,j+n+1}=\prod_{k=0}^n A_{j+k,j+k+1},
  $$
  as desired.
\end{proof}

\begin{proposition}\label{condiciones para BnL}
  Let $\mm$ and $a$ be as above. We have
  \begin{equation}\label{MYM}
    \mm Y^{n-1}\mm=a \sum_{j=0}^{n} Y^j \mm Y^{n-j}-a^2 Y^{n+1}
  \end{equation}
  and
  \begin{equation}\label{potencia se anula}
    (\mm-aY)^{n+1}=0.
  \end{equation}
  Moreover, set $m_i:=\mm_{i,i+1}$ and let $L:=(L_1=n,L_2,L_3,\dots)$ be the increasing sequence of integers such
  that $m_{L_i}\ne 0$ and $m_i=0$ if $L_k<i<L_{k+1}$  for some $k$. Then
  \begin{enumerate}
    \item $\prod_{k=1}^{n+1}(m_{j+k}-a)=0$ for all $j\ge 0$.
    \item If $m_j\ne 0$, then $m_{j+k}=0$ for $k=1,\dots,n-1$.
    \item We have $m_{L_i}=a$ and $(L_{i+1}-L_{i})\in\{n,n+1\}$ for all $i\ge 1$.
  \end{enumerate}
\end{proposition}

\begin{proof}
  From~\eqref{igualdad matricial para Y^n M tilde} and Lemma~\ref{mixtos 2} we obtain
  \begin{eqnarray*}
    Y^n \mm &=& d\left(Y^{n+1}+\mm Y^{n-1}\mm + \sum_{j=0}^n Y^j \mm Y^{n-j} \right)-\left(Y^{n}+\sum_{j=0}^{n-1} Y^j \mm Y^{n-1-j}\right)Y+a Y^{n+1}
    \\
    &=& (d-1) \sum_{j=0}^{n-1} Y^j \mm Y^{n-j}+d Y^n \mm+d\mm Y^{n-1}\mm+(d-1+a)Y^{n+1}.
  \end{eqnarray*}
  So
  \begin{align*}
    0&=d\mm Y^{n-1}\mm-(1-d) \sum_{j=0}^{n} Y^j \mm Y^{n-j}+(d-1+a)Y^{n+1}\\
    &=d\mm Y^{n-1}\mm-ad \sum_{j=0}^{n} Y^j \mm Y^{n-j}+(a-ad)Y^{n+1}
  \end{align*}
  where the second equality follows from $(1-d)=ad$.
  Using again $(1-d)=ad$ and dividing by $d$ we obtain~\eqref{MYM}. Multiplying by $a^{n-1}$ we obtain
  $$
    0=\mm (aY)^{n-1}\mm- \sum_{j=0}^{n} (aY)^j \mm (aY)^{n-j}+(aY)^{n+1}.
  $$
  But the right hand side is the expansion of $(aY-\mm)^{n+1}$, since in that expansion
  the only term with two times the factor $\mm$ is $\mm (aY)^{n-1}\mm$. This proves~\eqref{potencia se anula}.
  Item~(1) follows directly from~\eqref{potencia se anula} and Lemma~\ref{producto en la diagonal}, since
  $(\mm-aY)_{i,i+1}=m_i-a$.

  By Lemma~\ref{mixtos 2} we have $\mm Y^k\mm =0$ for $k=0,\dots,n-2$. But then
  $$
    0=\left(\mm Y^{k-1}\mm\right)_{j,j+k+1}=\mm_{j,j+1} \mm_{j+k,j+k+1}=m_j m_{j+k}
  $$
  for $k=1,\dots,n-1$, so item~(2) is true.

  Item~(2) implies that $L_{i+1}-L_i\ge n$. Assume by contradiction that $L_{i+1}-L_i> n+1$. Then
  $$
    \prod_{k=1}^{n+1}(m_{L_i+k}-a)=(-a)^{n+1}\ne 0,
  $$
  contradicting item~(1), hence $(L_{i+1}-L_{i})\in\{n,n+1\}$ for all $i\ge 1$.

  Finally, if $m_{L_i}\ne a$ for some $i>1$, then
  $$
        \prod_{k=1}^{n+1}(m_{L_i-2+k}-a)=(-a)^{n}(m_{L_i}-a)\ne 0,
  $$
  contradicting again item~(1), hence $m_{L_i}=a$ for all $i$.
\end{proof}

\section{The family $B(a,L)$ and quasi-balanced sequences}
\label{seccion familia}
In this section we will describe a family of twisted tensor products that arise in the case~(2) of Proposition~\ref{clasificacion}.
So $\sigma$ is a twisting map and $Y$ and $M$ are as in Corollary~\ref{graded condition} and $\mm=M-Y$.
Moreover there exists $n\in\mathds{N}$ with $n\ge 2$, such that $\mm_{k*}=\mm_{0*}$ for $1< k< n$ and
$\mm_{n*}=d E_0-E_1+a E_{n+1}$ for some $a,d\in K^{\times}$ with $d(a+1)=1$.

By Proposition~\ref{condiciones para BnL}(3), the twisting map defines a sequence $L=(L_1,L_2,\dots)$ such that $\mm_{L_i,L_i+1}=a$ and such that
$\mm_{k,k+1}=0$ if $L_i<k<L_{i+1}$ for some $i$.
Moreover $L_{j+1}-L_j\in\{n,n+1\}$ for all $j>1$, so $L$ belongs to the set of sequences
$$
  \Delta(n,n+1)=\{L\in\mathds{N}^{\mathds{N}}\ :\ L_1=n\text{ and }L_{j+1}-L_j\in\{n,n+1\}\text{ for all $j>1$}\}.
$$
We will construct an infinite matrix $M(L,a)$, where $L\in \Delta(n,n+1)$, $n\ge 2$ and $a\ne 0,1$; by setting
$d:=\frac{1}{1+a}$, and defining each row of the infinite matrix in the following way
$$
  M(L,a)_{j*}:=\begin{cases}
                E_0-E_1, & \mbox{if } j<L_1 \\
            d^k (E_0-E_1), & \mbox{if } L_k<j<L_{k+1} \text{ for some $k$} \\
            d^k E_0- d^{k-1}E_1+a E_{L_k+1}, & \mbox{if $j=L_k$} \text{ for some $k$.}
          \end{cases}
$$
  Note that $M(L,a)_{ij}=0$ if $j\notin\{0,1\} \cup \{L_k\}_{k\in\mathds{N}}$.
We will prove in Propositions~\ref{LL es necesario para BnL} and~\ref{LL es suficiente para BnL}  that
 the infinite matrix $\mm:=M(L,a)$ defines a twisting
map via Remark~\ref{condiciones de corolario pero para mm}, if and only if  $L$ is a quasi-balanced sequence, where the set of quasi-balanced
sequences $\mathcal{L}$ is defined as follows:
$$
  \mathcal{L}:= \{L\in\Delta(n,n+1),\ L_{r}-1\le L_j+L_{r-j}\le L_{r}\text{ for all $0<j<r$}\}.
$$
We first describe some properties of quasi-balanced sequences.

\begin{definition}\label{Def quasi balanced}
  We say that a sequence $L\in\Delta(n,n+1)$ is $r$-balanced, if $L_{r}=L_j+L_{r-j}$ for all $0<j<r$. We say that
  the sequence is $r$-quasi-balanced, if $L_{r}-1\le L_j+L_{r-j}\le L_{r}$ for all $0<j<r$. We also say that a finite sequence
  $(L_1,\dots,L_{r_0})$ is quasi-balanced, if $L_{r}-1\le L_j+L_{r-j}\le L_{r}$ for all $0<j<r\le r_0$.
\end{definition}
Note that every sequence $L\in\Delta(n,n+1)$ is trivially $1$-balanced, and it is also easy to see that
it is $2$-quasi-balanced and $3$-quasi-balanced. However the sequences beginning with
$(n,2n,3n+1,4n+2,\dots)$ or $(n,2n+1,3n+1,4n+1,\dots)$ are not $4$-quasi-balanced.

 Clearly a sequence
$L\in\Delta(n,n+1)$ is quasi-balanced (i.e. belongs to $\mathcal{L}$), if and only if it is $r$-quasi-balanced for all $r$, if and only if all
the sequences $L_{\le r}:=(L_1,\dots,L_r)$ are quasi-balanced.
For a given sequence $L$ and  $0<j<r $ we define $\Delta_{r,j}:=L_r-L_j-L_{r-j}$. The fact that $(L_1,\dots,L_{r_0})$ is quasi-balanced
is equivalent to the fact that $\Delta_{r,j}\in \{0,1\}$ for all $0<j<r\le r_0$.
\begin{proposition}\label{continuar balanceado}
  Assume that the finite sequence $(L_1,\dots,L_{r_0})$ is quasi-balanced. Then either
  $(L_1,\dots,L_{r_0},L_{r_0}+n)$ is quasi-balanced or $(L_1,\dots,L_{r_0},L_{r_0}+n+1)$ is quasi-balanced.
\end{proposition}

\begin{proof}
  We want to prove that either
  \begin{equation} \label{Delta menos}
    \Delta^{-}_{r_0+1,j}:=L_{r_0}+n-L_j-L_{r_0+1-j}\in\{0,1\}\ \text{for all $0<j<r_0+1$}
  \end{equation}
  or that
  \begin{equation}\label{Delta plus}
    \Delta^{+}_{r_0+1,j}:=\Delta^{-}_{r_0+1,j}+1\in\{0,1\}\ \text{for all $0<j<r_0+1$}.
  \end{equation}
  Note that
  $$
    \Delta^{-}_{r_0+1,j}=\Delta_{r_0,j}+n-(L_{r_0+1-j}- L_{r_0-j})\in\{\Delta_{r_0,j},\Delta_{r_0,j}-1\}\subset\{-1,0,1\},
  $$
  hence $\Delta^{+}_{r_0+1,j}\in\{0,1,2\}$.

  Now assume by contradiction that there exist $j_0,j_1$ such that $\Delta^{-}_{r_0+1,j_0}=-1$ and such that $\Delta^{+}_{r_0+1,j_1}=2$, i.e.,
  $\Delta^{-}_{r_0+1,j_1}=1$. We can assume that $j_1>j_0$ since $\Delta^{-}_{r_0+1,j}=\Delta^{-}_{r_0+1,r_0+1-j}$.
  Now we set $s:=j_1-j_0=(r_0+1-j_0)-(r_0+1-j_1)>0$ and obtain
  \begin{eqnarray*}
   2 &=& \Delta^{-}_{r_0+1,j_1}-\Delta^{-}_{r_0+1,j_0} \\
   &=& -L_{r_0+1-j_1}- L_{j_1}+L_{r_0+1-j_0}+ L_{j_0}+L_s-L_s \\
   &=& L_{r_0+1-j_0}-L_s-L_{r_0+1-j_1}-(L_{j_1}-L_{j_0}-L_s) \\
   &=&  \Delta(r+1-j_0,s)- \Delta(j_1,s).
  \end{eqnarray*}
  But $\Delta(r+1-j_0,s), \Delta(j_1,s)\in\{0,1\}$, so we obtain a contradiction, which proves that one of~\eqref{Delta menos}
  or~\eqref{Delta plus} is necessarily true. This concludes the proof.
\end{proof}

\begin{corollary}
  If $\tilde L=(L_1,\dots,L_r)$ is quasi-balanced, then there exists $L\in\mathcal{L}$ such that $L_{\le r}=\tilde L$.
\end{corollary}

\begin{proof}
  Construct inductively $L_{r+1}$, $L_{r+2}$,\dots using Proposition~\ref{continuar balanceado}.
\end{proof}

\begin{lemma}\label{no hay twisting}
  If $L\in \Delta(n,n+1)$, but $L\notin\mathcal{L}$, then there exists $k,r\in\mathds{N}$ such that either
  \begin{enumerate}
    \item $L_{k+r}=L_k+L_r-1$ and $L_{r-1}+2<L_r$ or
    \item $L_{k+r}=L_k+L_r+2$ and $L_r+2<L_{r+1}$.
  \end{enumerate}
\end{lemma}

\begin{proof}
  Let $m:=\min\{i: \ L_{\le i}\text{ is not quasi-balanced}\}$. For $0<j<m$ we have
  $$
    \Delta_{m,j}=\Delta_{m-1,j}+(L_{m}-L_{m-1})-(L_{m-j}-L_{m-j-1}).
  $$
  Since $\Delta_{m-1,j}\in\{0,1\}$, $(L_{m}-L_{m-1})\in\{n,n+1\}$ and $(L_{m-j}-L_{m-j-1})\in\{n,n+1\}$ necessarily
  $\Delta_{m,j}\in\{-1,0,1,2\}$. But $L_{\le m}$ is not quasi-balanced, so there exists $j$ such that either $\Delta_{m,j}=-1$ or $\Delta_{m,j}=2$.
  In the case $\Delta_{m,j}=-1$ set $r:=\min\{j:\ \Delta_{m,j}=-1\}$ and $k:=m-r$. Then $\Delta_{m,r-1}=0$, since
  $|\Delta_{m,r}-\Delta_{m,r-1}|\le 1$.
  But
  $$
    \Delta_{m,r-1}=\Delta_{m,r}+(L_{r}-L_{r-1})-(L_{m-r+1}-L_{m-r}),
  $$
  hence $0=-1+(L_{r}-L_{r-1})-(L_{m-r+1}-L_{m-r})$, and since $(L_{r}-L_{r-1}),(L_{m-r+1}-L_{m-r})\in\{n,n+1\}$, we have
  $L_{r}-L_{r-1}=n+1$ and $L_{m-r+1}-L_{m-r}=n$. So $L_{r}-L_{r-1}=n+1>2$ and together with $\Delta_{k+r,r}=L_{k+r}-L_k-L_r=-1$ this proves
  that we are in the case of~(1).

  On the other hand, in the case $\Delta_{m,j}=2$ set $r:=\max\{j:\ \Delta_{m,j}=2\}$ and $k:=m-r$. Then $\Delta_{m,r+1}=1$, since
  $|\Delta_{m,r}-\Delta_{m,r+1}|\le 1$.
  But
  $$
    \Delta_{m,r+1}=\Delta_{m,r}+(L_{m-r}-L_{m-r-1})-(L_{r+1}-L_{r}),
  $$
  hence $1=2+(L_{m-r}-L_{m-r-1})-(L_{r+1}-L_{r})$, and since $(L_{m-r}-L_{m-r-1}),(L_{r+1}-L_{r})\in\{n,n+1\}$, we have
  $(L_{r+1}-L_{r})=n+1$ and $L_{m-r}-L_{m-r-1}=n$. So $L_{r+1}-L_{r}=n+1>2$ and together with $\Delta_{k+r,r}=L_{k+r}-L_k-L_r=2$ this yields~(2) and
  concludes the proof.
\end{proof}

\begin{proposition}\label{LL es necesario para BnL}
  Let $L\in\Delta(n,n+1)$. If the infinite matrix $\mm:=M(L,a)$ defines a twisting map via Remark~\ref{condiciones de corolario pero para mm}, then
  $L\in\mathcal{L}$.
\end{proposition}

\begin{proof}
  Assume that $\mm=M(L,a)$ defines a twisting map and assume by contradiction that $L\in \Delta(n,n+1)\setminus\mathcal{L}$. In the first case of
  Lemma~\ref{no hay twisting} note that
  $L_{r-1}<L_r-2<L_r-1<L_r$, hence by definition $\mm_{L_r-2,*}=\mm_{L_r-1,*}=d^{r-1}(E_0-E_1)$, and so, by Lemma~\ref{mixtos}(1) we
  have $\mm Y^{L_r-2}\mm=0$. But
  then, using that $L_{k+r}=L_k+L_r-1$ and that $a,d,1-d\ne 0$ we obtain
  \begin{eqnarray*}
      0=(\mm Y^{L_r-2}\mm)_{L_k,0} &=& d^k \mm_{L_r-2,0}-d^{k-1} \mm_{L_r-1,0}+a \mm_{L_k+L_r-1,0}\\
    &=& d^{k+r-1}-d^{k+r-2}+a d^{k+r}\\
    &=& d^{r+k-2}(d-1+ad^2)= d^{r+k-2}(-ad+ad^2)\\
    &=&-ad^{k+r-1}(1-d)\ne 0,
  \end{eqnarray*}
  a contradiction which discards this case. On the other hand, in the second case of Lemma~\ref{no hay twisting} note that
  $L_{r}<L_r+1<L_r+2<L_{r+1}$, hence by definition $\mm_{L_r+1,*}=\mm_{L_r+2,*}=d^{r}(E_0-E_1)$, and so, by Lemma~\ref{mixtos}(1) we have
  $\mm Y^{L_r+1}\mm=0$. But
  then
  $$
    0=(\mm Y^{L_r+1}\mm)_{L_k,L_{r+k}+1}=d^k \mm_{L_r+1,L_{r+k}+1}-d^{k-1} \mm_{L_r+2,L_{r+k}+1}+a \mm_{L_k+1+L_r+1,L_{r+k}+1}.
  $$
  But $\mm_{L_r+1,L_{r+k}+1}= \mm_{L_r+2,L_{r+k}+1}=0$ and $L_{k+r}=L_k+L_r+2$, and so we arrive at
  $$
    0= a  \mm_{L_{r+k},L_{r+k}+1}=a^2\ne 0,
  $$
  a contradiction that discards the second case of Lemma~\ref{no hay twisting} and concludes the proof.
\end{proof}

\section{Existence of twisting maps for the family $B(a,L)$}
\label{seccion 10}
In Theorem~\ref{anda} we determined the form of the matrices $\mm$ corresponding to twisting maps of the family $A(n,d,a)$ and in
Corollary~\ref{Corolario anda}  we proved conversely  that each such matrix defines
actually a twisting map. Similarly, in Proposition~\ref{LL es necesario para BnL} we determined the form of the matrices $\mm$
corresponding to twisting maps of the family $B(a,L)$, namely, that necessarily $L\in \mathcal{L}$.
This last section is devoted to the proof that this condition is sufficient.
So, along this last section, we assume that $L\in\Delta(n,n+1)$ and we set
$$
\mm:=M(L,a).
$$
We will prove that if
$L\in\mathcal{L}$, then the matrix $\mm$ defines a twisting map via  Remark~\ref{condiciones de corolario pero para mm}.
\begin{proposition}\label{lema conditions to prove BNL}
  Assume that the following conditions are satisfied
  \begin{enumerate}
    \item $\mm Y^{k}\mm=0$ for $k\in \mathds{N}_0$ with $k,k+1\ne L_t$ for all $t$,
    \item $\mm Y^{L_r}\mm+Y\mm Y^{L_r-1}\mm=a Y^{L_r+1}\mm$ for all $r\in\mathds{N}$,
    \item $\mm Y^{L_r-1}\mm=a\sum_{i=0}^{L_r}Y^i\mm Y^{L_r-i}-ra^2Y^{L_r+1}$ for all $r\in\mathds{N}$.
  \end{enumerate}
  Set $L_0:=0$, then for all $r\in\mathds{N}_0$ and $L_r\le j<L_{r+1}$, we have
  \begin{equation}\label{potencia de M}
    d^r M^{j+1}=\sum_{i=0}^{j}Y^i \mm Y^{j-i}+(1-ra)Y^{j+1}.
  \end{equation}
\end{proposition}

\begin{proof}
  We will prove~\eqref{potencia de M} by induction on $j$. Note that it holds trivially for $j=L_0=0$.
  Now assume that~\eqref{potencia de M} is true for $j$ and we will prove that it holds for $j+1$. We have to consider two cases for this inductive
  step: either $L_r\le j<L_{r+1}-1$ for some $r\ge0$, or $j=L_r-1$ for some $r\ge 1$.

  First assume that $L_r\le j<L_{r+1}-1$. Then
  we multiply~\eqref{potencia de M} by $M=Y+\mm$ and obtain
  \begin{eqnarray*}
    d^r M^{j+2} & =& \left(\sum_{i=0}^{j}Y^i \mm Y^{j-i}+(1-ra)Y^{j+1}\right)(Y+\mm)\\
    & = &\sum_{i=0}^{j}Y^{i} \mm Y^{j+1-i}+(1-ra)Y^{j+2}+(1-ra) Y^{j+1}\mm +\sum_{i=0}^{j} Y^i \mm Y^{j-i}\mm\\
     & = &\sum_{i=0}^{j}Y^{i} \mm Y^{j+1-i}+ Y^{j+1}\mm+(1-ra)Y^{j+2}-ra Y^{j+1}\mm +\sum_{i=0}^{j} Y^{j-i} \mm Y^{i}\mm.
  \end{eqnarray*}
  We want to compute the last sum and note that for $i=0,\dots,j$ there are three possibilities:\vspace{1mm}\\
  1. $i=0$ or $L_t<i<L_{t+1}-1$ for some $0\le t <j$, and then $Y^{j-i}\mm Y^i \mm=0$ by condition~(1),\vspace{1mm}\\
  2. $i= L_t-1$ for some $1\le t\le r$, or \vspace{1mm}\\
  3. $i= L_t$ for some $1\le t\le r$. \vspace{1mm}\\
  Hence we have
  $$
  \sum_{i=0}^{j} Y^{j-i} \mm Y^{i}\mm=\sum_{t=1}^r \left(Y^{j-L_1+1} \mm Y^{L_t-1}\mm+ Y^{j-L_t} \mm Y^{L_t}\mm\right).
  $$
  But by condition~(2) we have
  $$
    Y \mm Y^{L_t-1}\mm +\mm Y^{L_t}\mm =a Y^{L_t+1}\mm\quad\text{for $1\le t \le r$,}
  $$
  and so, multiplying by $Y^{j-L_t}$, we obtain
  $$
    Y^{j-L_t+1} \mm Y^{L_t-1}\mm +Y^{j-L_t}\mm Y^{L_t}\mm =a Y^{j+1}\mm,
  $$
  and then
  $$
    \sum_{i=0}^{j} Y^{j-i} \mm Y^{i}\mm=\sum_{t=1}^{r}\left(Y^{j-L_t+1} \mm Y^{L_t-1}\mm +Y^{j-L_t}\mm Y^{L_t}\mm\right)=ra\, Y^{j+1}\mm
  $$
  which implies~\eqref{potencia de M} for $j+1$:
  $$
    d^r M^{j+2}=\sum_{i=0}^{j+1}Y^{i} \mm Y^{j+1-i} +(1-ra)Y^{j+2}.
  $$

  Now, if $j=L_{r}-1$ for some $r\ge 1$, then $L_{r-1}< j< L_r$ and by the same argument as before, from~\eqref{potencia de M}  we obtain
  $$
    d^{r-1} M^{j+2}= \sum_{i=0}^{j+1}Y^{i} \mm Y^{j+1-i}+(1-(r-1)a)Y^{j+2}-(r-1)a Y^{j+1}\mm +\sum_{i=0}^{j} Y^{j-i} \mm Y^{i}\mm,
  $$
  and we also obtain
  $$
    \sum_{i=0}^{j-1} Y^{j-i} \mm Y^{i}\mm=(r-1)a Y^{j+1}\mm.
  $$
  Consequently
  $$
    d^{r-1} M^{j+2}=\sum_{i=0}^{j+1}Y^{i} \mm Y^{j+1-i}+ (1-(r-1)a)Y^{j+2}+\mm Y^{j}\mm,
  $$
  which gives
  $$
    d^{r-1} M^{L_r+1}=\sum_{i=0}^{L_r}Y^{i} \mm Y^{L_r-i}+ (1-(r-1)a)Y^{L_r+1}+\mm Y^{L_r-1}\mm.
  $$
  Inserting into this equality the value of $\mm Y^{L_r-1}\mm$ according to condition~(3) we obtain
  $$
    d^{r-1} M^{L_r+1}=(1+a)\sum_{i=0}^{L_r}Y^{i} \mm Y^{L_r-i}+ (1+a)(1-ra)Y^{L_r+1},
  $$
  and using that $d(a+1)=1$ we obtain~\eqref{potencia de M} for $j+1=L_r$. This completes the inductive step and concludes the proof.
\end{proof}

\begin{proposition}\label{corolario conditions to prove BNL}
  Assume the hypotheses of
Proposition~\ref{lema conditions to prove BNL}.  Then~\eqref{fundamental matrix equality for mm} is valid for all $k \in\mathds{N}_0$.
\end{proposition}

\begin{proof}
  A straightforward computation shows that~\eqref{fundamental matrix equality for mm} holds for $k=0,1$.
 Moreover, for $k>0$ we have by definition
 $$
  \mm_{k*}:=\begin{cases}
            d^r (E_0-E_1), & \mbox{if } L_r<k<L_{r+1} \text{ for some $r\ge 0$} \\
            d^r E_0- d^{r-1}E_1+a E_{L_r+1}, & \mbox{if $k=L_r$} \text{ for some $r\ge 1$.}
          \end{cases}
$$
Hence, if  $L_r<k<L_{r+1}$ for some $r\ge 0$, then~\eqref{fundamental matrix equality for mm} reads
$$
Y^k \mm=\sum_{j= 0}^{k+1} \mm_{kj}M^{k+1-j} Y^j= d^r M^{k+1}-d^r M^k Y=d^r M^k \mm,
$$
and so we have to prove
\begin{equation}\label{por probar en caso general}
  Y^k \mm= d^r M^k \mm, \ \text{for all $r\ge 0$ and all $L_r<k<L_{r+1}$.}
\end{equation}
On the other hand, if $k=L_r$ for some $r>0$,
then~\eqref{fundamental matrix equality for mm} reads
$$
Y^k \mm=\sum_{j= 0}^{k+1} \mm_{kj}M^{k+1-j} Y^j= d^r M^{k+1}-d^{r-1} M^k Y+a Y^{k+1},
$$
and so we have to prove
\begin{equation}\label{por probar en caso k es Lr}
  Y^{L_r} \mm= d^r M^{L_r+1}-d^{r-1} M^{L_r} Y+a Y^{L_r+1}, \ \text{for all $r\ge 1$.}
\end{equation}
We first prove inductively~\eqref{por probar en caso general}. We know that it holds for $k=1$, and we will prove that if it holds for
$k$ with $L_r<k<L_{r+1}-1$ for some $r\ge 0$, then it holds for $k+1$, and that if it holds for $k=L_r-1$ for some $r\ge 1$, then it holds for
$k+2=L_r+1$.

  Let $k$ satisfy $L_r<k<L_{r+1}-1$ for some $r\ge 0$. From condition~(1) of Proposition~\ref{lema conditions to prove BNL}, we obtain
  $$
    M Y^{k}\mm=\mm Y^k \mm+Y^{k+1}\mm =Y^{k+1}\mm.
  $$
  Hence, if $Y^k\mm=d^r M^k\mm$ for such $k$, then $Y^{k+1}\mm=d^r M^{k+1}\mm$, which is~\eqref{por probar en caso general} for $k+1$.

  Assume now that $k=L_r-1$ for some $r\ge 1$, and that~\eqref{por probar en caso general} holds, i.e.,
  $$
  Y^{L_r-1}\mm=d^{r-1} M^{L_r-1}\mm
  $$
  Multiplying this by $M^2=\mm Y+Y\mm+Y^2$ from the left we obtain
  $$
    \mm Y^{L_r}\mm+Y \mm Y^{L_r-1}\mm+Y^{L_r+1}\mm=d^{r-1}M^{L_r+1}\mm
  $$
  and combined with $\mm Y^{L_r}\mm+Y \mm Y^{L_r-1}\mm=a Y^{L_r+1}\mm$, which holds by condition~(2) of
  Proposition~\ref{lema conditions to prove BNL}, this yields
  $$
    d^{r-1}M^{L_r+1}\mm=(a+1)Y^{L_r+1}\mm.
  $$
  Using $(a+1)d=1$ this gives $Y^{L_r+1}\mm=d^r M^{L_r+1}\mm$, which is~\eqref{por probar en caso general} for $k+2=L_r+1$.

  Finally we prove~\eqref{por probar en caso k es Lr}, which is~\eqref{fundamental matrix equality for mm} for $k=L_r$. For this consider the equality
  $$
    d^{r} M^{L_r+1}=\sum_{i=0}^{L_r}Y^i \mm Y^{L_r-i}+(1-ra)Y^{L_r+1},
  $$
  which is the equality~\eqref{potencia de M} for $j=L_r$, and the equality
  $$
    d^{r-1} M^{L_r}Y=\sum_{i=0}^{L_r-1}Y^i \mm Y^{L_r-i}+(1-(r-1)a)Y^{L_r+1},
  $$
  which is the equality~\eqref{potencia de M} for $j=L_r-1$, multiplied by $Y$ from the right.
  Subtracting the second equality from the first, we obtain
  \begin{eqnarray*}
     d^{r} M^{L_r+1}- d^{r-1} M^{L_r}Y  &=& \sum_{i=0}^{L_r}Y^i \mm Y^{L_r-i}-\sum_{i=0}^{L_r-1}Y^i \mm Y^{L_r-i}-a Y^{L_r+1} \\
    &=& Y^{L_r}\mm-a Y^{L_r+1},
  \end{eqnarray*}
  which is~\eqref{por probar en caso k es Lr}.  Hence~\eqref{fundamental matrix equality for mm} holds for $k=L_r$, concluding the proof.
\end{proof}

In order to prove that for $L\in\mathcal{L}$ the matrix $\mm:=M(L,a)$ defines a twisting map, we decompose the matrix $\mm$ into three summands.
Set
$$
m_i:=
     \begin{cases}
      1, & \mbox{if } i\le n=L_1 \\
      d^r, & \mbox{if } L_r< i\le L_{r+1}
     \end{cases}
$$
and define the infinite matrix $M_1$ by $(M_1)_{ij}\coloneqq \delta_{0j} m_i$. Now define the set
$$
   |L|\coloneqq \{L_t\}_{t\in \mathds{N}}=\{L_1,L_2,\dots,L_k,\dots\},
$$
and set
$$
n_i=
\begin{cases}
  a, & \mbox{if } i\in |L| \\
  0, & \mbox{otherwise}.
\end{cases}
$$
We define the infinite matrix $B$ by $B_{ij} \coloneqq \delta_{ij} n_i$. Then
 $\mm=BY+YM_1-M_1 Y$ and
 \begin{equation}\label{B se relaciona con M1}
   BYM_1+Y M_1-M_1=0.
 \end{equation}
 \begin{lemma} \label{productos parciales}
   The following equalities hold for all $k\in \mathds{N}_0$:
   \begin{enumerate}
     \item $M_1 Y^k M_1= m_k M_1$,
     \item $M_1 Y^k B=n_k M_1 Y^k$,
     \item $(B Y^k M_1)_{ij}= \delta_{0j} n_i m_{k+i}$,
     \item $(B Y^k B)_{ij}= \delta_{i+k,j}(n_i n_{i+k})$.
   \end{enumerate}
 \end{lemma}
\begin{proof}
  A straightforward computation.
\end{proof}
\begin{remark} \label{BYkB se anula}
  Note that if $L\in \mathcal{L}$ and $k,k+1\notin |L|$, then $B Y^{k+1}B=0$. In fact, by Lemma~\ref{productos parciales}(4),
  it suffices to prove that if $n_i n_{i+k+1}\ne 0$ for some $i$, then either $k\in |L|$ or $k+1\in |L|$. But $n_i n_{i+k+1}\ne 0$ implies
  $i =L_t$ and $i+k+1=L_{r+t}$ for some $r,t\in \mathds{N}$.
  Since $L$ is quasi-balanced, then either $L_{r+t}=L_t+L_r$, which implies $k+1=L_r$, or $L_{r+t}=L_t+L_r+1$,
  which implies $k=L_r$.
\end{remark}
\begin{lemma} \label{caso mixto}
  Let $L\in\mathcal{L}$. If $i\in |L|$ and $k\notin |L|$, then $m_{i+k+1}=dm_i m_k$.
\end{lemma}
\begin{proof}
  We know that $i=L_t$ for some $t\in \mathds{N}$.

  If $k<n=L_1$, then $L_t<i+k+1\le L_t+n\le L_{t+1}$, and so $m_{i+k+1}=d^{t}$.
  Since $m_i=d^{t-1}$ and $m_k=1$, this proves the result in this case.
  Else $L_r<k< L_{r+1}$ for some $r\in\mathds{N}$, and since $L$ is quasi-balanced, we have
  $$
   L_{r+t}-1\le L_t+L_r<i+k<L_t+L_{r+1}\le L_{t+r+1}.
  $$
  Hence
  $$
    L_{r+t}< i+k+1 \le L_{t+r+1},
  $$
  and so $m_{i+k+1}=d^{r+t}$. Since $m_i=d^{t-1}$ and $m_k=d^r$, this concludes the proof.
\end{proof}

\begin{lemma}
  Assume $L\in \mathcal{L}$, and let $i,k\in\mathds{N}_0$ and $r\in \mathds{N}$. Then
  \begin{align}
    & n_i m_{i+k+1}+m_k m_{i+1}-m_{k+1} m_{i} = n_k m_{i+k+1}, \label{mi ni Delta}\\
    & n_i n_{L_r+i+1}+n_{i+1} n_{L_r+i+1} = a n_{L_r+i+1}, \label{mi ni salto}\\
    & a \sum_{j=0}^{L_r} n_{i+j}  =ra^2 + n_i n_{L_r+i}. \label{mi ni sumatoria}
  \end{align}
\end{lemma}

\begin{proof}
  We first prove~\eqref{mi ni Delta} in each of the four possible cases. Note that for $k\in |L|$ we have $m_{k+1}=dm_k$, and for $k\notin |L|$ we
  have   $m_{k+1}=m_k$.\vspace{1mm}\\
  {\bf{Case $i,k\notin |L|$:}} \
   Here $n_i=n_k=0$, $m_{i+1}=m_i$, $m_{k+1}=m_k$, hence $m_k m_{i+1}-m_{k+1} m_{i}=0$ and both sides
  of~\eqref{mi ni Delta} vanish.\vspace{1mm}\\
  {\bf{Case $i\in |L|, k\notin |L|$:}}\
   Here $n_i=a$, $n_k=0$, $m_{i+1}=dm_i$, $m_{k+1}=m_k$ and by Lemma~\ref{caso mixto} we have
   $m_{i+k+1}=dm_i m_k$. Hence
   $$
     n_i m_{i+k+1}+m_k m_{i+1}-m_{k+1} m_{i}= m_im_k(ad+d-1)=0,
   $$
  and both sides
  of~\eqref{mi ni Delta} vanish.\vspace{1mm}\\
  {\bf{Case $i\notin |L|, k\in |L|$:}}\
   Here $n_i=0$, $n_k=a$, $m_{i+1}=m_i$, $m_{k+1}=dm_k$ and by Lemma~\ref{caso mixto} we have
   $dm_i m_k=m_{i+k+1}$. Hence
   $$
     n_i m_{i+k+1}+m_k m_{i+1}-m_{k+1} m_{i}= (1-d)m_i m_k=ad m_i m_k =a m_{i+k+1}=n_k m_{i+k+1},
   $$
  as desired.\vspace{1mm}\\
  {\bf{Case $i,k\in |L|$:}}\
   Here $n_i=a$, $n_k=a$, $m_{i+1}=dm_i$ and $m_{k+1}=dm_k$. Hence
   $$
     n_i m_{i+k+1}+m_k m_{i+1}-m_{k+1} m_{i}=a m_{i+k+1}=n_k m_{i+k+1},
   $$
  as desired, concluding the proof of~\eqref{mi ni Delta}\vspace{1mm}.

  Now we prove~\eqref{mi ni salto}. If $L_r+i+1\notin |L|$, then both sides vanish. If $L_r+i+1\in |L|$, then
  we have to prove that $n_i+n_{i+1}=a$. There exists $t>0$ such that $L_r+i+1=L_{r+t}$. Since \vspace{1mm} $L\in \mathcal{L}$,
  we have either $L_{r+t}=L_r+L_t$ or $L_{r+t}=L_r+L_t+1$. \vspace{1mm}\\
  In the first case $L_t=i+1$, $n_i=0$ and $n_{i+1}=a$; and in the second case $L_t=i$, $n_i=a$ and $n_{i+1}=0$. Hence in both cases
  $n_i+n_{i+1}=a$, which proves~\eqref{mi ni salto}.

  In order to prove~\eqref{mi ni sumatoria} we consider three cases.
  \begin{itemize}
      \item[-] If $i<L_1=n$, then $L_r\le L_r+i< L_{r+1}$, and so
               $$
                 \sum_{j=0}^{L_r} n_{j+i}=\sum_{s=1}^{r}n_{L_s}=ar.
               $$
               Since $n_i=0$, we obtain~\eqref{mi ni sumatoria}.
      \item[-] If $L_t<i<L_{t+1}$, then $L_{r+t}\le L_r+i< L_{r+t+1}$, since $L\in\mathcal{L}$ implies
               $$
                L_{r+t}-1\le L_r+L_t < L_r+i<L_r+L_{t+1}\le L_{r+t+1}.
               $$
                Hence
               $$
                 \sum_{j=0}^{L_r} n_{j+i}=\sum_{j=i}^{L_r+i} n_{j}=\sum_{s=t+1}^{t+r} n_{L_s}=ar,
               $$
                and using $n_i=0$, we obtain~\eqref{mi ni sumatoria} in this case.
      \item[-] If $i=L_t$ for some $t$, then $a n_{i+L_r}=n_i n_{L_r+i}$, and so it suffices to prove
                $$
                  ra=  \sum_{j=0}^{L_r-1} n_{j+i}=\sum_{j=L_t}^{L_r+L_t-1} n_j.
                $$
                But $L\in \mathcal{L}$ implies
                $$
                L_{r+t-1}\le L_r+L_{t-1}+1\le L_r+L_t-1 < L_{r+t},
                $$
                and so
                $$
                \sum_{j=L_t}^{L_r+L_t-1} n_j=\sum_{s=0}^{r-1}n_{L_{t+s}}=ra,
                $$
                as desired.
  \end{itemize}
  Thus~\eqref{mi ni sumatoria} holds in all cases, concluding the proof.
\end{proof}

In order to verify the conditions of Proposition~\ref{lema conditions to prove BNL} we need to compute
\begin{equation} \label{expresion primera de mm Yk mm}
\mm Y^k \mm= \mm Y^k BY+ \mm Y^{k+1}M_1 -\mm Y^{k}M_1 Y,
\end{equation}
and so we have to compute $\mm Y^k M_1$.

\begin{lemma} \label{mm Yk M1}
  If $L\in \mathcal{L}$, then we have
  $$
    \mm Y^k M_1=n_k Y^{k+1} M_1.
  $$
\end{lemma}

\begin{proof}
  We have
  $$
    \left(n_k Y^{k+1} M_1\right)_{ij}=\delta_{0j}n_k m_{i+k+1}
  $$
  and by Lemma~\ref{productos parciales} and equality~\eqref{mi ni Delta}, we also have
  \begin{align*}
    \left(\mm Y^k M_1\right)_{ij}&=\left(BY^{k+1}M_1+YM_1 Y^k M_1-M_1 Y^{k+1} M_1\right)_{ij} \\
    & =    \delta_{0j}\left(n_i m_{i+k+1}+m_k m_{i+1}-m_{k+1} m_{i}\right)\\
    &= \delta_{0j}n_k m_{i+k+1},
  \end{align*}
  as desired.
\end{proof}

\begin{proposition} \label{expresion de mm Yk mm}
  If $L\in \mathcal{L}$, then
  $$
   \mm Y^k \mm=BY^{k+1}BY+ n_k Y M_1 Y^{k+1}-n_{k+1}M_1 Y^{k+2}+n_{k+1}Y^{k+2}M_1-n_k Y^{k+1} M_1 Y.
  $$
\end{proposition}

\begin{proof}
  By Lemma~\ref{mm Yk M1} and equality~\eqref{expresion primera de mm Yk mm} we have
  $$
   \mm Y^k \mm=\mm Y^{k}BY+n_{k+1}Y^{k+2}M_1-n_k Y^{k+1} M_1 Y.
  $$
  From Lemma~\ref{productos parciales}(2) we obtain
  $$
    \mm Y^{k}BY=B Y^{k+1}B Y+Y M_1 Y^k BY-M_1 Y^{k+1}BY=BY^{k+1}BY+ n_k Y M_1 Y^{k+1}-n_{k+1}M_1 Y^{k+2},
  $$
  which concludes the proof.
\end{proof}

\begin{proposition}\label{LL es suficiente para BnL}
  For each $a\in K\setminus\{0,-1\}$ and $L\in\mathcal{L}$, the matrix $\mm=M(L,a)$ defines a twisting map via
  Remark~\ref{condiciones de corolario pero para mm}.
\end{proposition}

\begin{proof}
  Since $\mm_{0j}=\delta_{0j}-\delta_{1j}$ and $\mm_{kj}=0$ for $j>k+1$, by Remark~\ref{condiciones de corolario pero para mm} and
  Propositions~\ref{lema conditions to prove BNL} and~\ref{corolario conditions to prove BNL}, it suffices to check the
conditions~(1)--(3) in Proposition~\ref{lema conditions to prove BNL}.

  If $k,k+1\notin |L|$, then by Proposition~\ref{expresion de mm Yk mm} we know that $\mm Y^{k}\mm=BY^{k+1}BY$. By
  Remark~\ref{BYkB se anula} we also know that $BY^{k+1}B=0$ for $k,k+1\notin |L|$, which proves item~(1).

  In order to prove item~(2), we use Proposition~\ref{expresion de mm Yk mm} and compute
  \begin{align*}
    \mm Y^{L_r}\mm   &= B Y^{L_r+1} B +a Y M_1 Y^{L_r+1}-a Y^{L_r+1}M_1 Y, \\
    Y \mm Y^{L_r-1}\mm &= Y B Y^{L_r}B Y-a Y M_1 Y^{Lr+1}+a Y^{L_r+2} M_1, \\
    a Y^{L_r+1} \mm &= a Y^{L_r+1} BY+ a Y^{L_r+2} M_1 -a Y^{L_r+1}M_1 Y.
  \end{align*}
  So we have to prove
  $$
    BY^{L_r+1}B+Y B Y^{L_r}B=a Y^{L_r+1}B.
  $$
  We have
  \begin{align*}
   (BY^{L_r+1}B)_{ij} & =\sum_k B_{i,k}B_{k+L_r+1,j}=\sum_k \delta_{ik}\delta_{k+L_r+1,j} n_i n_{k+L_r+1}=\delta_{i+L_r+1,j}n_i n_{i+L_r+1}, \\
   (YBY^{L_r}B)_{ij} & =\sum_k B_{i+1,k} B_{k+L_r,j}=\sum_k \delta_{i+1,k}\delta_{k+L_r,j} n_{i+1} n_{k+L_r}=\delta_{i+L_r+1,j}n_{i+1} n_{i+L_r+1}, \\
   (a Y^{L_r+1}B)_{ij} & = aB_{i+L_r+1,j}=a \delta_{i+L_r+1,j} n_{i+L_r+1},
  \end{align*}
  and so~\eqref{mi ni salto} concludes the proof of item~(2).

  In order to prove item~(3), we compute
  \begin{align*}
    a \sum_{i=0}^{L_r} Y^i \mm Y^{L_r-i} & = a \sum_{i=0}^{L_r} Y^{i} B Y^{L_r+1-i}+a \sum_{i=0}^{L_r} Y^{i+1} M_1 Y^{L_r-i}
    -a \sum_{i=0}^{L_r} Y^i M_1 Y^{L_r+1-i} \\
    &= a \sum_{i=0}^{L_r} Y^{i} B Y^{L_r+1-i}+ a Y^{L_r+1}M_1-a M_1 Y^{L_r+1}.
  \end{align*}
  Since by Proposition~\ref{expresion de mm Yk mm} we know that
  $$
    \mm Y^{L_r-1}\mm =B Y^{L_r}BY-a M_1 Y^{L_r+1}+a Y^{L_r+1}M_1,
  $$
  we have to prove
  $$
    a \sum_{s=0}^{L_r} Y^s B Y^{L_r+1-s}= B Y^{L_r}BY +ra^2 Y^{L_r+1}.
  $$
  But
  \begin{align*}
    (Y^s B Y^{L_r+1-s})_{ij} &= B_{s+i,j+s-L_r-1}=\delta_{i,j-L_r-1} n_{s+i}, \\
    (BY^{L_r}BY)_{ij} &= \sum_k B_{i,k} B_{k+L_r,j-1} =\sum_k \delta_{ik}\delta_{k+L_r,j-1} n_i n_{k+L_r}=n_i n_{i+L_r}\delta_{i+L_r,j-1}, \\
    ra^2 (Y^{L_r+1})_{ij} &= ra^2 \delta_{i,j-L_r-1},
  \end{align*}
  hence it suffices to prove
  $$
    a\sum_{s=0}^{L_r} n_{s+i}=n_i n_{i+L_r}+ra^2,
  $$
  which holds by~\eqref{mi ni sumatoria}. This concludes the proof.
\end{proof}

\begin{remark}
  Our strategy contains two main components.  On one hand
  the approach of equalities of infinite matrices yields conditions that reduce the possibilities to very few families.
  Even in the complicated case ~(2) of  Proposition~\ref{clasificacion}  one can achieve the classification of
  all possible twisting maps up to any degree, with increasing amount of computational work.
  On the other hand
  proving that a given infinite matrix yields a twisting map requires to verify an infinite number of matrix equalities for infinite matrices.
  We are able to realize this difficult task in Corollary~\ref{Corolario anda} and in Proposition~\ref{LL es suficiente para BnL}.
  In the first case we only have to prove one of the equalities,
  since those twisting maps have the $n$-extension property, i.e., they
  are completely determined by the values of $\mm_{k,*}$ for $k\le n$. In the case of Proposition~\ref{LL es suficiente para BnL} we manage to
  decompose the infinite matrix into three simpler ones, and we prove the required matrix equalities using properties of these simpler matrices.

  Notice that none of the twisting maps constructed in Proposition~\ref{LL es suficiente para BnL} has the $m$-extension property for any $m$.
  This is a direct consequence of the
  following property of quasi-balanced sequences:

  Let $L_{\le r}=(L_1,\dots,L_r)$ be a quasi-balanced partial sequence. Then there exists an extension of $L_{\le r}$ of the form
  $(L_1,\dots,L_r,\dots,L_{r+k})$ such that
  both
  $$
    (L_1,\dots,L_r,\dots,L_{r+k},L_{r+k}+n)\quad\text{and}\quad (L_1,\dots,L_r,\dots,L_{r+k},L_{r+k}+n+1)
  $$
  are quasi-balanced partial sequences.
\end{remark}

In a forthcoming article this property will be proven, together with several other properties of these sequences. For example,
the quasi-balanced sequences show a surprising connection to Euler's totient function and so they are interesting on its own.

There are several open problems related to the results of this article, we want to highlight two of them:
\begin{enumerate}
  \item In computations not shown in this paper we have found 16 different cases for the first $4n-1$ rows of the matrices
  $\mm$ corresponding to the case~(2) of Proposition~\ref{clasificacion}, and 4 of these cases correspond to
  twisting maps of the family $B(a,L)$.
  Does there exist any twisting map corresponding to any of the other 12 cases?
  \item Does any twisting map related to the case~(2) of Proposition~\ref{clasificacion} has the $m$-extension property for any $m$?
\end{enumerate}

\noindent \textbf{Acknowledgement.} We thank the anonymous referee for the thorough revision and numerous helpful suggestions.

\end{document}